\documentclass[10pt]{amsart}
\usepackage{bm}

\usepackage{amsmath}
\usepackage{amsthm}
\usepackage{amssymb}

\usepackage{hyperref}

\newcommand{\R}{\mathbb{R}}
\newcommand{\lr}[1]{\langle #1 \rangle}
\newcommand{\Lr}[1]{\Big\langle #1 \Big\rangle}
\newcommand{\eps}{\varepsilon}
\newcommand{\wt}[1]{\widetilde{#1}}
\newcommand{\wh}[1]{\widehat{#1}}

\newcommand{\uh}{u^{\rm{hyp}}}
\newcommand{\uhp}{u^{\rm{hyp},+}}
\newcommand{\uhn}{u^{\rm{hyp},-}}
\newcommand{\whp}{w^{\rm{hyp},+}}

\newcommand{\ue}{u^{\rm{ell}}}
\newcommand{\uep}{u^{\rm{ell},+}}

\newtheorem{thm}{Theorem}[section]
\newtheorem{prop}[thm]{Proposition}

\newtheorem{lem}[thm]{Lemma}
\newtheorem{cor}[thm]{Corollary}

\theoremstyle{remark}

\DeclareMathOperator{\supp}{supp}
\allowdisplaybreaks[1]

\numberwithin{equation}{section}

\date{\today}

\title[Asymptotics of solutions to the short-pulse eq]{Large time asymptotics of solutions\\ to the short-pulse equation}

\author[M. Okamoto]{Mamoru Okamoto}
\address{Division of Mathematics and Physics, Faculty of Engineering, Shinshu University, 4-17-1 Wakasato, Nagano City 380-8553, Japan}
\email{m\_okamoto@shinshu-u.ac.jp}

\thanks{This work was supported by JSPS KAKENHI Grant number JP16K17624.}



\date{\today}

\begin{document}

\begin{abstract}
We consider the long-time behavior of solutions to the short-pulse equation.
Using the method of testing by wave packets, we prove small data global existence and modified scattering.
\end{abstract}

\maketitle

\section{Introduction}

We consider the Cauchy problem for the short-pulse equation
\begin{equation} \label{sp}
\begin{aligned}
& u_{tx} = u + (u^3)_{xx}, \\
& u(0,x) = u_0(x) ,
\end{aligned}
\end{equation}
where $u = u(t,x) : \R_+ \times \R \rightarrow \R$ is an unknown function, and $u_0$ is a given function.
The short-pulse equation gives an approximate solution to Maxwell's equation describing the propagation of ultra-short optical pulses in nonlinear media (see \cite{SchWay04}).

We consider the previous results for the generalized Ostrovsky equation
\begin{equation} \label{Ostrovsky}
u_{tx} = (u^{p})_{xx}, \quad
u(0,x) = u_0(x)
\end{equation}
for $p \in \mathbb{N}_{\ge 2}$.
The case in which $p=2$ is known as the Ostrovsky-Hunter equation \cite{Boy05} or the short-wave equation \cite{Hun90}.
Pelinovsky and Sakovich \cite{PS10} showed global well-posedness in the energy space for $p=3$ and small initial data.
Stefanov et al. \cite{SSK10} showed local existence of a unique solution to \eqref{Ostrovsky} with $u_0 \in H^s(\R)$ when $s>\frac{3}{2}$.
They also proved global existence and scattering for $p\ge 4$ and small initial data $u_0 \in H^5(\R) \cap W^{3,1}(\R)$.
To confirm the global existence of a solution, we need to consider the smallness of the initial data.
Liu et al. \cite{LPS09, LPS10} demonstrated wave-breaking phenomena at $p=2,3$, and, in particular, the existence of a blowing-up solution.
Hayashi et al. \cite{HNN13} (see also \cite{HayNau14}) provided the $L^{\infty}$ decay estimates and the solution scatters to a free solution for $p \ge 4$ and small initial data $u_0 \in H^s(\R) \cap \dot{H}^{-1}(\R)$ with $s>2$ and $x \partial_x u_0 \in L^2(\R)$.
In \cite{HNN14}, they also proved the nonexistence of the usual scattering states for $p=3$.
Recently, Niizato \cite{Nii14} showed the existence of a modified scattering state of \eqref{sp} for small initial data in $u_0 \in H^s(\R) \cap \dot{H}^{-1}(\R)$ with $s>10$ and $x \partial_x u_0 \in H^5(\R)$.
Using the factorization technique, Hayashi and Naumkin \cite{HayNau15} proved the existence of a modified scattering state for \eqref{sp}, for a larger class of initial data than that in \cite{Nii14}.
In \cite{HayNau15}, they took the initial data that satisfy $u_0 \in H^s(\R) \cap \dot{H}^{-1}(\R)$ and $x \partial_x u_0 \in H^r (\R)$ with $s> \frac{5}{2} +r$ and $r>\frac{3}{2}$.
However, it appears that more regularity for the initial data is needed (see Appendix \ref{appendix}).

In this paper, we use the method of testing by wave packets based on the work of Ifrim and Tataru \cite{IfrTat15, IfrTat16} (see also \cite{HarG16, HIT}).
This method in some sense interpolates between the physical and the Fourier side analysis of an asymptotic equation.
Instead of localizing on either the physical or the Fourier side, we use a mixed wave packet style phase space localization.
We prove small data global existence and modified scattering in a large class of initial data.

Let $L$ denote the linear operator of \eqref{sp}:
\[
L := \partial_t - \partial_x^{-1}.
\]
We note that
\[
\partial_x^{-1}f (x) = \frac{1}{\sqrt{2\pi}} \int_{-\infty}^x f(y) dy
\]
holds provided that $f \in \dot{H}^{-1} (\R)$, where $\partial_x^{-1} := \mathcal{F}^{-1} \frac{1}{i\xi} \mathcal{F}$.
To obtain pointwise estimates for the solutions, we use the vector field
\[
J := x-t \partial_x^{-2} ,
\]
which satisfies $J = e^{t \partial_x^{-1}} x e^{-t \partial_x^{-1}}$.
This is a powerful tool for studying the large time existence of nonlinear evolution equations (see \cite{Kla85, HNN13, HNN14, Nii14, HayNau15} and references therein).
Factorizing the symbol $x+\frac{t}{\xi^2}$ of $J$, we define
\[
J_{\pm} := \sqrt{|x|} \mp i \sqrt{t} \partial_x^{-1} .
\]
Here $J_+$ is hyperbolic on positive frequencies and elliptic on negative frequencies.
These operators are useful in our analysis.

The equation \eqref{sp} is invariant under the scaling transformation
\begin{equation} \label{scaling}
u (t,x) \mapsto \lambda^{-1} u(\lambda^{-1} t, \lambda x)
\end{equation}
for any $\lambda >0$.
The generator of the scaling transformation is given by
\[
S := -t \partial_t + x \partial_x -1 ,
\]
which is related to $L$ and $J$ as follows:
\[
S = -t L + J \partial_x -1 .
\]

The free solution for \eqref{sp} is written as follows:
\[
e^{t \partial_x^{-1}} f (x) = (\mathcal{F}^{-1} [e^{\frac{t}{i\xi}}] \ast f ) (x) , \quad
\mathcal{F}^{-1} [e^{\frac{t}{i\xi}}] (x) = \frac{1}{\sqrt{2\pi}} \int_{\R} e^{i ( x \xi - \frac{t}{\xi})} d\xi .
\]
Because $\partial_{\xi} ( x \xi - \frac{t}{\xi}) = x + \frac{t}{\xi^2}$ becomes zero if and only if $\xi = \pm \sqrt{\frac{t}{|x|}}$ and $x<0$, the stationary phase method implies that  the free solution $e^{t \partial_x^{-1}} f (x)$ decays rapidly when $x>0$ and oscillates when $x<0$.
As the solution to \eqref{sp} with small initial data behaves like the free solution, this observation shows that modified scattering occurs when $x<0$.

To state our main result, we introduce the norm with respect to the spatial variable
\[
\| u (t) \|_{X^s} := \left( \| u (t) \|_{H^s}^2 + \| u (t) \|_{\dot{H}^{-1}}^2 + \| J \partial_x u (t) \|_{L^2}^2 \right)^{\frac{1}{2}}
\]
for $s \in \R$.

\begin{thm} \label{thm}
Let $s >4$.
Assume that the initial data $u_0$ at time $0$ satisfies
\[
\| u_0 \|_{X^s} \le \eps \ll 1.
\]
Then, there exists a unique global solution $u$ that satisfies the bound
\[
\| u(t) \|_{X^s} \lesssim \eps \lr{t}^{C \eps} ,
\]
as well as the pointwise bound
\begin{equation} \label{est:u_infty}
\| u(t) \|_{L^{\infty}} + \| u_x (t) \|_{L^{\infty}} \lesssim \eps \lr{t}^{-\frac{1}{2}}.
\end{equation}
Furthermore, there exists a unique modified final state $W \in L^{\infty}(\R_-)$ such that, for large $t \ge 1$,
\begin{align*}
u(t,x) = & \frac{2}{\sqrt{t}} \bm{1}_{\R_-} (x) \Re \left\{ W \left( \frac{x}{t} \right) \exp \left( -2i \sqrt{t|x|} + 3i \sqrt{\frac{t}{|x|}} \left| W \left( \frac{x}{t} \right) \right|^2 \log t \right) \right\} \\
& + O \left( \eps t^{-\frac{1}{2} -\kappa} \right)
\end{align*}
holds uniformly with respect to $x \in \R$, where $0<\kappa < \min \left\{ \frac{1}{4} - \frac{5}{8} \alpha^{\ast}, \frac{s-2}{2(s+1)} \right\}$ and $\alpha^{\ast} := \min \left\{ \frac{2}{45}, \frac{2}{2s+1}, \frac{2(s-4)}{3(s+1)} \right\}$.
\end{thm}

We note that our initial data space has the norm
\[
\| u_0 \|_{X^s} \sim \| u_0 \|_{H^s} + \| u_0 \|_{\dot{H}^{-1}} + \| x \partial_x u_0 \|_{L^2} .
\]
Accordingly, modified scattering holds for a larger class of initial data than shown by previous results.
In particular, we does not need the regularity of $x \partial_x u_0$.

We do not focus here on the upper bound of $\kappa$.
The crucial point is that the decay of the remainder part is faster than $t^{-\frac{1}{2}}$, which is the decay rate of the free solution.
In fact, Stefanov et al. proved the dispersive estimate
\[
\| e^{t \partial_x^{-1}} u_0 \| _{L^p} \lesssim t^{-(\frac{1}{2}-\frac{1}{p})} \| u_0 \|_{\dot{W}^{\frac{3}{2}-\frac{3}{p}, p'}}
\]
for $2<p<\infty$ (Theorem 3 in \cite{SSK10}).
Setting $p=\infty$ formally, we expect the decay rate of the $L^{\infty}$ norm of the free solution to be $t^{-\frac{1}{2}}$.

Roughly speaking, we will show the bound
\[
\| u_x(t) \|_{L^{\infty}} \lesssim t^{-\frac{1}{2}} \| u(t) \|_{H^s}^{\frac{1}{2}} \| J \partial_x u(t) \|_{L^2}^{\frac{1}{2}}
\]
for $s>4$ (see Proposition \ref{prop:est|u|} below), which implies \eqref{est:u_infty}.
Here, the assumption $s>4$ is almost optimal from the viewpoint of the scaling invariance.
Indeed, the fraction
\[
\frac{t^{\frac{1}{2}} \| u_x (t) \|_{L^{\infty}}}{\| u(t) \|_{\dot{H}^4}^{\frac{1}{2}} \| J \partial_x u(t) \|_{L^2}^{\frac{1}{2}}}
\]
is invariant under the scaling transformation \eqref{scaling}.

The remainder of this paper is organized as follows.
In \S \ref{S:energy}, we show the energy estimates and the existence of the local in time solution to \eqref{sp}.
In \S \ref{S:KS_type}, we prove a priori estimates, which give the pointwise bounds.
In \S \ref{S:wave_packet}, we construct a wave packet and observe its properties.
In \S \ref{S:proof}, by combining the estimates proved in previous sections, we prove our main theorem.
In Appendix \ref{appendix}, we provide a remark on the paper by Hayashi and Naumkin \cite{HayNau15}.

Finally, in this section, we present the notations used throughout this paper.
We denote the space of all smooth and compactly supported functions on $\R$ by $C_0^{\infty} (\R)$.
We denote the space of all rapidly decaying functions on $\R$ by $\mathcal{S}(\R)$.
We define the Fourier transform of $f$ by $\mathcal{F}[f]$ or $\widehat{f}$.
We use the inhomogeneous Sobolev spaces $H^s(\R)$ with the norm $\| f \|_{H^s}:= \| \lr{\cdot}^s \widehat{f} \| _{L^2}$, where $\lr{\xi} := (1+ | \xi |^2)^{\frac{1}{2}}$.
We also use the homogeneous Sobolev norm $\| f \|_{\dot{H}^s} := \| |\cdot |^s \wh{f} \|_{L^2}$.

In estimates, we use $C$ to denote a positive constant that can change from line to line.
If $C$ is absolute, or depends only on parameters that are considered fixed, we often use $X \lesssim Y$ in place of $X \le CY$.
We then use $X \ll Y$ to denote $X \le C^{-1} Y$ and $X \sim Y$ to denote $C^{-1} Y \le X \le C Y$.
We write $X = Y + O(Z)$ when $|X-Y| \lesssim Z$.

Let $\delta>0$ be a small constant, which is needed only to demonstrate Proposition \ref{prop:gamma_decay}.
For concreteness, we take $\delta = \frac{1}{1000}$.
Let $\sigma \in C_0^{\infty}(\R )$ be an even function with $0 \le \sigma \le 1$ and $\sigma (\xi ) = \begin{cases} 1, & \text{if } |\xi | \le 1, \\ 0, & \text{if } |\xi| \ge 2^{\delta} .\end{cases}$
For any $R, \, R_1, R_2 >0$ with $R_1 < R_2$, we set 
\begin{align*}
\sigma _R (\xi) := \sigma \Big( \frac{\xi}{R} \Big) -\sigma \Big( \frac{2^{\delta} \xi}{R} \Big) , \quad \sigma _{\le R} (\xi) := \sigma \Big( \frac{\xi}{R} \Big) , \quad \sigma _{>R} (\xi) := 1- \sigma _{\le R} (\xi) ,\\
\sigma _{<R}(\xi) := \sigma_{\le R}(\xi) - \sigma _{R} (\xi) , \quad
\sigma_{R_1 \le \cdot \le R_2} (\xi) := \sigma _{\le R_2}(\xi) - \sigma_{<R_1}(\xi) .
\end{align*}
For any $N, \, N_1, N_2 \in 2^{\delta \mathbb{Z}}$ with $N_1 < N_2$, we define
\[
P_N f := \mathcal{F}^{-1} [\sigma_{N} \widehat{f}] , \quad
P_{N_1 \le \cdot \le N_2} := \mathcal{F}^{-1}[\sigma_{N_1 \le \cdot \le N_2} \wh{f}] .
\]
We denote the characteristic function of an interval $I$ by $\bm{1}_{I}$.
For $N \in 2^{\delta \mathbb{Z}}$, we define the Fourier multipliers with the symbols $\bm{1}_{\R_+}(\pm \xi)$ and $\sigma_N(\xi) \bm{1}_{\R_+} (\pm \xi )$ by $P^{\pm}$ and $P_N^{\pm}$, respectively.

\section{Energy estimates} \label{S:energy}

The results in this section were essentially proved in \cite{SSK10} (see also \cite{SchWay04, HNN13}).
For completeness, we give an outline of this proof.

First, we recall the energy estimate proved by Stefanov et al. \cite{SSK10}.

\begin{lem}[Lemma 1 in \cite{SSK10}] \label{SSK}
Let $u$ be a smooth solution of the equation
\[
u_{tx} = u + F(t,x) u_{xx} + G(t,x)
\]
for $t>0$, where $F$ and $G$ are smooth functions.
Then, for every $s>1$, we have
\begin{align*}
\partial_t \| u(t) \|_{\dot{H}^s}^2
\lesssim & \| \partial_x F(t, \cdot ) \|_{L^{\infty}} \| u(t) \|_{\dot{H^s}}^2 \\
&  + \| u(t) \|_{\dot{H}^s} ( \| G(t,\cdot ) \|_{\dot{H}^{s-1}} + \| \partial_x u (t) \|_{L^{\infty}} \| F(t,\cdot ) \|_{\dot{H}^s}).
\end{align*}
\end{lem}

A simple calculation yields the following equations:
\[
[L,J]=0, \quad
[L,S]=-L, \quad
[S,\partial_x]=-\partial_x, \quad
S(fg) = Sf \cdot g + fSg + fg .
\]

\begin{lem} \label{lem:energy}
Let $s>1$.
Let $u$ be a solution to \eqref{sp} in a time interval $[0,T]$ satisfying
\[
\| u_0 \|_{X^s} \le \eps \ll 1
\]
and assume that there exists a constant $D$ with $1<D \le \eps^{-1}$ such that
\[
\| u(t) \|_{L^{\infty}} + \| u_x (t) \|_{L^{\infty}} \le D \eps \lr{t}^{-\frac{1}{2}}.
\]
Then,
\[
\| u(t) \|_{X^s} \le 10 \eps \lr{t}^{D_{\ast} \eps} ,
\]
where $D_{\ast} \lesssim D$.
\end{lem}

\begin{proof}
Integration by parts yields
\[
\partial_t \| u(t) \|_{L^2} = 2 \int_{\R} u (\partial_x^{-1} u + \partial_x(u^3)) dx =0 .
\]
Similarly, we have
\begin{align*}
\partial_t \| u(t) \|_{\dot{H}^1}^2
& = 6 \int_{\R} u_x( u^2 u_{xx} + 2uu_x^2) dx = 6 \int_{\R} u u_x^3 dx \\
& \le 6 \| u(t) \|_{L^2}^2 \| u(t) \|_{L^{\infty}} \| u_x(t) \|_{L^{\infty}} .
\end{align*}
From the integral equation for \eqref{sp}, we have
\begin{equation} \label{eq:uH-1}
\begin{aligned}
\| u(t) \|_{\dot{H}^{-1}}
& \le \| u_0 \|_{\dot{H}^{-1}} + \int_0^t \| u(t')^3 \| _{L^2} dt' \\
& \le \| u_0 \|_{\dot{H}^{-1}} + \int_0^t \| u(t') \|_{L^{\infty}}^2 \| u(t') \|_{L^2} dt' .
\end{aligned}
\end{equation}
For higher order derivatives, we apply Lemma \ref{SSK} with $F(t,x)=3u(t,x)^2$ and $G(t,x)=6u(t,x) u_x(t,x)^2$.
Note that
\begin{align*}
& \| \partial_x F(t, \cdot ) \|_{L^{\infty}}
\le 6 \| u(t) \|_{L^{\infty}} \| u_x(t) \|_{L^{\infty}} , \quad
\| F(t, \cdot ) \|_{\dot{H}^s}
\lesssim \| u(t) \|_{L^{\infty}} \| u(t) \|_{\dot{H}^s}, \\
& \| G(t, \cdot ) \|_{\dot{H}^{s-1}}
\lesssim \| u(t) \|_{L^{\infty}}^2 \| u (t) \|_{H^s} + \| u(t) \|_{L^{\infty}} \| u_x(t) \|_{L^{\infty}} \| u(t) \|_{H^{s-1}} .
\end{align*}
Thus, by combining the above estimates and Lemma \ref{SSK} with Gronwall's inequality, we obtain
\[
(\| u(t) \|_{H^s}^2 + \| u(t) \|_{\dot{H}^{-1}}^2)^{\frac{1}{2}} \le \eps \lr{t}^{D_{\ast}\eps} .
\]

From $S= -t L + J \partial_x -1$ and
\[
\| Lu(t) \|_{L^2} = \| \partial_x(u^3)(t) \|_{L^2}
\le 3 \| u(t) \|_{L^{\infty}} \| \partial_x u(t) \|_{L^{\infty}} \| u(t) \|_{L^2}
\le 3 D^2 \eps^3 \lr{t}^{-1} ,
\]
the estimate of $\|J \partial_x u \|_{L^2}$ is reduced that of $\| S u\| _{L^2}$.
By
\begin{equation} \label{eq:Su}
\begin{aligned}
\partial_t \| S u (t) \|_{L^2}^2
& = 2 \int_{\R} Su \cdot LSu dx
= 2 \int_{\R} Su \cdot (S-1) \partial_x(u^3) dx \\
& = 2 \int_{\R} Su \cdot \partial_x(S-2) (u^3) dx
= 6 \int_{\R} Su \cdot \partial_x (u^2 Su) dx \\
& = 6 \int_{\R} u \partial_xu (Su)^2 dx
\le \| u(t) \|_{L^{\infty}} \| \partial_x u(t) \|_{L^{\infty}} \| Su(t) \|_{L^2}^2 ,
\end{aligned}
\end{equation}
Gronwall's inequality yields
\[
\| Su(t) \|_{L^2} \le 2 \eps \lr{t}^{D_{\ast}\eps} ,
\]
which concludes the proof.
\end{proof}

\begin{cor} \label{cor:LWP}
Let $s>\frac{3}{2}$ and $u_0 \in X^s$.
Then, there exists an existence time $T= T(\| u_0 \|_{H^s})$ and a unique solution $u$ to \eqref{sp} satisfying $\sup_{0 \le t \le T} \| u(t) \|_{X^s} \le 10 \| u_0 \|_{X^s}$.
\end{cor}

\begin{proof}
Set $u^{(0)} := u_0$ and for $n \in \mathbb{N}$, define
\[
\partial_{tx} u^{(n)} = u^{(n)} + 3 \partial_x \{ (u^{(n-1)})^2 \partial_x u^{(n)}\},
\quad u^{(n)} (0,x) = u_0 (x).
\]
In the same way as the proof of Lemma \ref{lem:energy}, we have
\[
\sup _{0 \le t \le T} \| u^{(n)} (t) \|_{H^s} \le \| u_0 \|_{H^s} \exp \left( C_1 T \sup _{0 \le t \le T} \| u^{(n-1)} (t) \|_{H^s}^2 \right)
\]
because $\| u(t) \|_{L^{\infty}} + \| u_x(t) \| _{L^{\infty}} \lesssim \| u(t) \|_{H^s}$.
Accordingly, by setting
\[
T := \frac{\log 2}{10 C_1 \| u_0 \|_{H^s}^2},
\]
we confirm that $\sup _{0 \le t \le T} \| u^{(n-1)} (t) \|_{H^s} \le 2 \| u_0 \|_{H^s}$ implies $\sup _{0 \le t \le T} \| u^{(n)} (t) \|_{H^s} \le 2 \| u_0 \|_{H^s}$.
Because $\sup _{0 \le t \le T} \| u^{(0)} (t) \|_{H^s} \le 2 \| u_0 \|_{H^s}$ holds, we have the bounded sequence $\{ u^{(n)} \}$ in $L^{\infty} ([0,T]; H^s(\R))$.
By the standard argument, we obtain a solution $u$ as the limit of the (sub)sequence.
Then, \eqref{eq:uH-1} and \eqref{eq:Su} yield
\begin{align*}
&
\begin{aligned}
\sup_{0 \le t \le T} \| u(t) \|_{\dot{H}^{-1}}
& \le \| u_0 \|_{\dot{H}^{-1}} + C_1 T \sup _{0 \le t \le T} \| u (t) \|_{H^s}^3 \\
& \le \| u_0 \|_{\dot{H}^{-1}} + \| u_0 \|_{H^s}
\le 2 \| u_0 \|_{X^s},
\end{aligned}
\\
&
\begin{aligned}
\sup _{0 \le t \le T} \| Su (t) \|_{L^2}
& \le \| (Su) (0) \|_{L^2} \exp \left( C_1 T \sup _{0 \le t T} \| u(t) \|_{H^s}^2 \right) \\
& \le 2 \| (Su) (0) \|_{L^2}
\le 4 \| u_0 \|_{X^s}.
\end{aligned}
\end{align*}
From $S= -t L + J \partial_x -1$, the solution $u$ is in $L^{\infty}([0,T]; X^s)$ and satisfies
\[
\sup_{0 \le t \le T} \| u(t) \|_{X^s} \le 10 \| u_0 \|_{X^s}.
\]

To show uniqueness, we take two solutions $u, u'$ to \eqref{sp}.
The calculation used in the proof of Lemma \ref{lem:energy} yields
\begin{align*}
& \sup _{0 \le t \le T} \| u(t) - u'(t) \|_{H^{s}} \\
& \lesssim \| u(0) - u'(0)\|_{H^s} \exp \left( C_2 T \sup_{0 \le t \le T} (\| u (t) \|_{H^{s+1}} + \| u' (t) \|_{H^{s+1}})^2 \right) .
\end{align*}
Hence, the solution is unique as a limit of classical solutions.
\end{proof}

\section{Pointwise decay estimates} \label{S:KS_type}

We decompose $u$ into positive and negative frequencies:
\[
u = u^+ + u^-, \quad u^{\pm} := P^{\pm} u.
\]
Because $u$ is real valued, $u^+ = \overline{u^-}$ and $u= 2 \Re u^+$.
Moreover,
\[
\| u^+ (t) \|_{X^s} = \| u^- (t) \|_{X^s} = \frac{1}{\sqrt{2}} \| u (t) \|_{X^s}.
\]
We write $u_N := P_N u$ and $u_N^+ := P_N^+ u$.
From $J \partial_x u_N = P_N (J \partial_x u) - \mathcal{F}^{-1} [ (\sigma_N)' \xi \wh{u}]$, we have
\[
\| u (t) \| _{X^s} \sim \left( \sum_{N \in 2^{\delta \mathbb{Z}}} \|u_N (t) \| _{X^s}^2 \right)^{\frac{1}{2}}.
\]

For $t \ge 1$, we further decompose $u^+$ into its hyperbolic and elliptic parts
\[
\uhp = \sum _{\substack{N \in 2^{\delta \mathbb{Z}} \\ N \le t}} \uhp_{N}, \quad
\uep = u^+ - \uhp ,
\]
where, for $N \le t$, we define
\[
\uhp_{N} := \sigma_N^{\rm{hyp}} u^+_{N}, \quad
\uep_{N} := u^+_{N} - \uhp_{N}.
\]
Here, $\sigma_N^{\rm{hyp}} (t,x) := \sigma_{\frac{1}{3} \frac{t}{N^2} \le \cdot \le 3 \frac{t}{N^2}} (x) \bm{1}_{\R_-} (x)$.

We note that $\uhp$ is supported in $\{ \frac{-x}{t} \ge \frac{1}{3 \cdot 2^{\delta}} t^{-2} \}$.
For $(t,x) \in \R^2$ with $-xt \ge \frac{1}{3 \cdot 2^{\delta}}$, the number of scaled dyadic numbers $2^{\delta \mathbb{Z}}$ satisfying $\frac{1}{3 \cdot 2^{\delta}} \frac{t}{N^2} \le |x| \le 3 \cdot 2^{\delta} \frac{t}{N^2}$ is  less than $\frac{5}{\delta}$.
Hence, $\uhp (t,x)$ is a finite sum of $\uhp_N(t,x)$.

The functions $\uhp_N$ and $\uep_N$ are frequency localized near $N$ in the following sense.

\begin{lem} \label{lem:freq_spat_loc}
For $2 \le p \le \infty$, any $a ,\, b ,\, c \in \R$ with $a \ge 0$ and $a+c \ge 0$, and any $R>0$, we have
\[
\| (1-P^+_{\frac{N}{2^{\delta}} \le \cdot \le 2^{\delta} N}) | \partial_x| ^{a} ( |x|^{b} \sigma_R P_N^+ f) \|_{L^p}
\lesssim_{a,b,c} N^{-c+\frac{1}{2}-\frac{1}{p}} R^{-a+b-c}  \| P_N^+ f \|_{L^2} . 
\]
Moreover, we may replace $\sigma_R$ on the left hand side by $\sigma_{>R}$ if $a+c>b+1$ and $\sigma_{<R}$ if $a+c\ge 0$ and $b=0$.
\end{lem}

\begin{proof}
It is sufficient to show the case $p=2$, because the general case follows from the Gagliardo-Nirenberg inequality $\| f \|_{L^{\infty}} \lesssim \| f \|_{L^2}^{\frac{1}{2}} \| \partial_x f \|_{L^2}^{\frac{1}{2}}$ and the interpolation $L^p(\R) = (L^2(\R), L^{\infty} (\R))_{[1-\frac{2}{p}]}$.
We write
\begin{align*}
& \mathcal{F} [|x|^{b} \sigma_R P_N^+ f](\xi)
= \int_{\R} |\xi-\eta|^{-(a+c)} \mathcal{F}[ |\partial_x|^{a+c} ( |x|^{b} \sigma_R)] (\xi -\eta ) \wh{P_N^+ f}(\eta) d\eta ,
\\
& |x|^{b} \sigma_R (x) = R^{b} (| \cdot |^{b} \sigma_1) \Big( \frac{x}{R} \Big) .
\end{align*}
Because $|\xi| \le |\eta| + |\xi - \eta|$ and $|\xi -\eta| \ge 2^{-2\delta} ( 2^{\delta} -1) N$ if $\xi \notin [\frac{N}{2^{2\delta}}, 2^{2\delta}N]$ and $\eta \in \supp \wh{P_N^+ f}$,
Young's inequality yields
\begin{align*}
& \| (1-P^+_{\frac{N}{2^{\delta}} \le \cdot \le 2^{\delta}N}) | \partial_x | ^{a} ( |x|^{b} \sigma_R P_N^+ f) \|_{L^2} \\
& \lesssim N^{-c} R^{-a+b-c} \| \mathcal{F}[ | \partial_x|^{a+c} (|x|^{b} \sigma_1)] \|_{L^1} \| P_N^+ f \|_{L^2} \\
& \lesssim N^{-c} R^{-a+b-c} \| P_N^+ f \|_{L^2} .
\end{align*}
The same calculation is valid when we replace $\sigma_R$ with $\sigma_{<R}$ if $a+c \ge 0$ and $b=0$.
From $\sigma_{>R} = \sum_{k=1}^{\infty} \sigma_{2^{k\delta}R}$, we can replace $\sigma_R$ on the left hand side by $\sigma_{>R}$ because the summation with respect to $k$ converges if $a+c>b+1$.
\end{proof}

The next proposition plays crucial role in our analysis.

\begin{prop} \label{prop:est|u|}
For $s > \frac{3}{2}$ and $0<t<1$, we have
\[
|u(t,x)| , \, |u_x(t,x)| \lesssim \| u (t) \|_{X^s}.
\]
For $s> \frac{5}{2}$ and $t \ge 1$, we have
\begin{align*}
& | \uhp (t,x) | \lesssim t^{-\frac{1}{2}} \min \left\{ \left( \frac{|x|}{t} \right)^{\frac{s}{4}-\frac{1}{2}}, \left( \frac{|x|}{t} \right)^{-\frac{3}{4}} \right\} \| u (t) \|_{X^s}, \\
& | \uhp_x (t,x) | \lesssim t^{-\frac{1}{2}} \min \left\{ \left( \frac{|x|}{t} \right) ^{\frac{s}{4}-1}, \left( \frac{|x|}{t} \right) ^{-\frac{5}{4}} \right\} \| u (t) \|_{X^s},
\end{align*}
and
\begin{align*}
& | \ue (t,x) | \lesssim t^{-\frac{2s-1}{2s+2}} \left( 1+ \log t \right) \| u (t) \|_{X^s}, \\\
& | \ue_x (t,x) | \lesssim t^{-\frac{2s-3}{2s+2}} \left( 1+ \log t \right) \| u (t) \|_{X^s}.
\end{align*}
\end{prop}

\begin{proof}
For $0<t<1$ and $s >\frac{3}{2}$, Sobolev's inequality yields
\begin{align*}
|u(t,x)| + |u_x(t,x)|
& \lesssim \sum _{N \in 2^{\delta \mathbb{Z}}} (1+N) \| u_N (t) \|_{L^{\infty}} \\
& \lesssim \sum_{N \in 2^{\delta \mathbb{Z}}} (1+N)N^{\frac{1}{2}} \| u_{N} (t) \|_{L^2}
\lesssim \| u (t) \| _{H^s} .
\end{align*}

For the high frequency case $N > t \ge 1$, we note that $u= \ue$ or $\uh=0$, because of the frequency restriction of $\uh$.
The calculation used above yields
\begin{align*}
& |u(t,x)|
\lesssim \sum _{\substack{N \in 2^{\delta \mathbb{Z}} \\ N > t}} \| u_N (t) \| _{L^{\infty}}
\lesssim \sum_{\substack{N \in 2^{\delta \mathbb{Z}} \\ N > t}} N ^{\frac{1}{2}} \| u_{N} (t) \|_{L^2}
\lesssim t^{-\frac{2s-1}{2}} \| u (t) \| _{H^s}, \\
& |u_x(t,x)|
\lesssim \sum _{\substack{N \in 2^{\delta \mathbb{Z}} \\ N > t}} N \| u_N (t) \| _{L^{\infty}}
\lesssim \sum_{\substack{N \in 2^{\delta \mathbb{Z}} \\ N > t}} N ^{\frac{3}{2}} \| u_{N} (t) \|_{L^2}
\lesssim t^{-\frac{2s-3}{2}} \| u (t) \| _{H^s}.
\end{align*}

Next, we consider the case of $t \ge 1$ and $N \le t$.
This is the main focus of our work.

\begin{lem} \label{lem:he_freq_est}
For $t \ge 1$ and $N \le t$, we have
\begin{align*}
& \| J_+ \partial_x \uhp _{N} (t) \|_{L^2} \lesssim t^{-\frac{1}{2}} N ( \| u_N (t) \|_{L^2} + \| J \partial_x u _N (t) \|_{L^2}), \\
& \Big\| \Lr{N^2 \frac{x}{t}} \ue_{N} (t) \Big\|_{L^2} \lesssim t^{-1} N ( \| u_{N} (t) \|_{L^2} + \| J \partial_x u_{N} (t) \|_{L^2}).
\end{align*}
\end{lem}

\begin{proof}
For the hyperbolic estimate, we use the equation
\begin{equation} \label{eq:J-f}
\Big\| \sqrt{\frac{|x|}{t}} \partial_x f \Big\|_{L^2}^2 + \| f \|_{L^2}^2
= t^{-1} \| J_- \partial_x f \|_{L^2}^2 + 2\Im \int_{\R} \sqrt{\frac{|x|}{t}} f(x) \partial_x \overline{f}(x) dx.
\end{equation}
We apply this to $f=J_+ \partial_x \uhp_{N}$.
A direct calculation yields
\[
J_- \partial_x f
= J_- \partial_x J_+ \partial_x \uhp_{N}
= -J \partial_x^2 \uhp_{N} - \frac{1}{2} \partial_x \uhp_N .
\]
Because
\begin{align*}
J \partial_x^2 \uhp_{N}
& = \sigma^{\rm{hyp}}_N P_{\frac{N}{2^{\delta}} \le \cdot \le 2^{\delta}N}^+ (J\partial_x^2 u_N) + \sigma^{\rm{hyp}}_N \mathcal{F}^{-1} [ -i (\sigma^+_{\frac{N}{2^{\delta}} \le \cdot \le 2^{\delta}N})' \xi ^2 \wh{u}_N] \\
& \quad + 2 \partial_x (x \partial_x \sigma^{\rm{hyp}}_N u_N^+) - (2 \partial_x \sigma^{\rm{hyp}}_N + x \partial_x^2 \sigma^{\rm{hyp}}_N ) u_N^+,
\end{align*}
Lemma \ref{lem:freq_spat_loc} implies the following:
\begin{equation} \label{eq:Jd^2}
\begin{aligned}
& \| J \partial_x^2 \uhp_N (t) \|_{L^2} \\
& \lesssim N ( \| u_N (t) \|_{L^2} + \| J \partial_x u_N (t) \|_{L^2}) + \| (1-P^+_{\frac{N}{2^{\delta}} \le \cdot \le 2^{\delta}N}) \partial_x (x \partial_x \sigma^{\rm{hyp}}_N u_N^+) (t) \|_{L^2} \\
& \quad + \| (1-P^+_{\frac{N}{2^{\delta}} \le \cdot \le 2^{\delta}N}) ( x \partial_x^2 \sigma^{\rm{hyp}}_N u_N^+ ) (t) \|_{L^2} \\
& \lesssim N ( \| u_N (t) \|_{L^2} + \| J \partial_x u_N (t) \|_{L^2}) .
\end{aligned}
\end{equation}
From $\partial_x \uhp_N = \partial_x \sigma^{\rm{hyp}}_N u_N^+ + \sigma^{\rm{hyp}}_N \partial_x u_N^+$, we have
\[
\| \partial_x \uhp_N (t) \|_{L^2}
\lesssim t^{-1} N^2 \| u_N (t) \|_{L^2} + N \| u_N (t) \|_{L^2}
\sim N \| u_N (t) \|_{L^2}
\]
This yields
\[
\| J_- \partial_x J_+ \partial_x \uhp_{N} (t) \|_{L^2}
\lesssim N (\| u_N (t) \|_{L^2} + \| J \partial_x u _N (t) \|_{L^2}) .
\]
The second expression on the right hand side of \eqref{eq:J-f} becomes
\begin{align*}
& \Im \int_{\R} \sqrt{\frac{|x|}{t}} J_+ \partial_x \uhp_N (t,x) \partial_x \overline{J_+ \partial_x \uhp_N} (t,x) dx \\
& = t^{-\frac{1}{2}} \Im \int_{\R} |x|^{\frac{1}{4}} J_+ \partial_x \uhp_N (t,x) \partial_x \left\{ |x|^{\frac{1}{4}} \overline{J_+ \partial_x \uhp_N} \right\} (t,x) dx \\
& = -t^{-\frac{1}{2}} \Re \int_{\R} \xi |\mathcal{F} [|x|^{\frac{1}{4}} J_+ \partial_x \uhp_N] (t,\xi)|^2 d\xi .
\end{align*}
From $|x|^{\frac{1}{4}} J_+ \partial_x \uhp_N = \partial_x (|x|^{\frac{3}{4}} \uhp_N) + \frac{3}{4} |x|^{-\frac{1}{4}} \uhp_N -i t^{\frac{1}{2}} |x|^{\frac{1}{4}} \uhp_N$ and Lemma \ref{lem:freq_spat_loc}, we have
\begin{align*}
& t^{-\frac{1}{4}} \| (1-P_{\frac{N}{2^{\delta}} \le \cdot \le 2^{\delta}N}^+) |\partial_x|^{\frac{1}{2}} |x|^{\frac{1}{4}} J_+ \partial_x \uhp_N (t) \|_{L^2} \\
& \lesssim t^{-1} N \| u_N (t) \|_{L^2}
\le t^{-\frac{1}{2}} N \| u_N (t) \|_{L^2}.
\end{align*}
Taking $f= J_- \partial_x \uhp_N$ in \eqref{eq:J-f}, we obtain the desired hyperbolic bound.

For the elliptic bound, we decompose $\ue$ into three parts $\ue = \sigma _{<\frac{1}{3} \frac{t}{N^2}} \ue_{N} + \sigma _{\frac{1}{3} \frac{t}{N^2} \le \cdot \le 3 \frac{t}{N^2}} \ue_N + \sigma _{>3 \frac{t}{N^2}} \ue_N$.
We observe that the equation
\begin{equation} \label{eq:ueob}
\left\| \frac{x}{t} f_{xx} \right\|_{L^2}^2 + \| f \|_{L^2}^2 = t^{-2} \| J \partial_x^2 f \|_{L^2}^2 - 2 \int_{\R} \frac{x}{t} |\partial_x f(x)|^2 dx
\end{equation}
holds for any smooth real valued function $f$.

From $(a+b)^2 \le (1+\delta) a^2+ (1+\delta^{-1}) b^2$ and Lemma \ref{lem:freq_spat_loc}, we have
\begin{align*}
& \left| \int_{\R} \frac{x}{t} | \partial_x ( \sigma _{>3 \frac{t}{N^2}} \ue_{N} ) (t,x)|^2 dx \right| \\
& \le \frac{N^2}{3} \left\| \frac{x}{t} \partial_x ( \sigma _{> 3\frac{t}{N^2}} \ue_{N} ) (t) \right\|_{L^2}^2 \\
& \le \frac{(1+\delta) 2^{4\delta}}{3} \left\| P_{\frac{N}{2^{\delta}} \le \cdot \le 2^{\delta}N} \partial_x \left( \frac{x}{t} \partial_x ( \sigma _{> 3 \frac{t}{N^2}} \ue_{N}) \right) (t) \right\|_{L^2}^2 \\
& \quad + C t^{-2}N^2 \| (1-P_{\frac{N}{2^{\delta}} \le \cdot \le 2^{\delta}N}) \left( x \partial_x (\sigma _{>3 \frac{t}{N^2}} \ue_{N}) \right) (t) \|_{L^2}^2 \\
& \le \frac{(1+\delta)^2 2^{4\delta}}{3} \left\| \frac{x}{t} \partial_x^2 \left( \sigma _{>3 \frac{t}{N^2}} \ue_{N} \right) (t) \right\|_{L^2}^2 + C t^{-2}N^2 \| u_{N} (t) \|_{L^2}^2 .
\end{align*}
Because $\sigma _{>3 \frac{t}{N^2}} \ue_{N} = \sigma _{>3 \frac{t}{N^2}} u_{N}$ and $\supp \partial_x \sigma_{> 3 \frac{t}{N^2}} \subset \{ |x| \sim t N^{-2} \}$, the calculation used in \eqref{eq:Jd^2} and Lemma \ref{lem:freq_spat_loc} yields
\[
\| J \partial_x^2 (\sigma _{> 3 \frac{t}{N^2}} \ue_{N}) (t) \|_{L^2}
\lesssim N ( \| u_N (t) \|_{L^2} + \| J \partial_x u_N (t) \|_{L^2} ) .
\]
Taking $f= \sigma _{>3 \frac{t}{N^2}} \ue_{N}$ in \eqref{eq:ueob}, and by $2 \frac{(1+\delta)^2 2^{4\delta}}{3} <1$, we have
\[
\left\| \frac{x}{t} \partial_x^2 (\sigma _{>3 \frac{t}{N^2}} \ue_{N}) (t) \right\|_{L^2}
\lesssim t^{-1} N ( \| u_N (t) \|_{L^2} + \| J \partial_x u_N (t) \|_{L^2} ) .
\]
Hence, by Lemma \ref{lem:freq_spat_loc}, we have
\begin{align*}
N^2 \Big \| \frac{x}{t} \sigma _{> 3 \frac{t}{N^2}} \ue_N (t) \Big\|_{L^2}
& \lesssim \Big \| P_{\frac{N}{2^{\delta}} \le \cdot \le 2^{\delta}N} \partial_x^2 \left( \frac{x}{t} \sigma _{> 3 \frac{t}{N^2}} \ue_N \right) (t) \Big\|_{L^2} \\
& \quad + t^{-1} N^2 \Big \| (1-P_{\frac{N}{2^{\delta}} \le \cdot \le 2^{\delta}N}) \left( x \sigma _{> 3 \frac{t}{N^2}} \ue_N \right) (t) \Big\|_{L^2} \\
& \lesssim \Big\| \frac{x}{t} \partial_x^2 \left( \sigma _{> 3 \frac{t}{N^2}} \ue_N \right) (t) \Big\|_{L^2} + t^{-1} N \| u_N (t) \|_{L^2} \\
& \lesssim t^{-1} N ( \| u_N (t) \|_{L^2} + \| J \partial_x u_N (t) \|_{L^2} ) .
\end{align*}

From $(a+b)^2 \le (1+\delta) a^2+ (1+\delta^{-1}) b^2$ and Lemma \ref{lem:freq_spat_loc}, we have
\begin{align*}
& \left| \int_{\R} \frac{x}{t} | \partial_x ( \sigma _{<\frac{1}{3} \frac{t}{N^2}} \ue_{N} ) (t,x)|^2 dx \right| \\
& \le \frac{1}{3 N^2} \| \partial_x ( \sigma _{<\frac{1}{3} \frac{t}{N^2}} \ue_{N} ) (t) \|_{L^2}^2 \\
& \le \frac{(1+\delta) 2^{4\delta}}{3} \| P_{\frac{N}{2^{\delta}} \le \cdot \le 2^{\delta}N} ( \sigma _{<\frac{1}{3} \frac{t}{N^2}} \ue_{N} ) (t) \|_{L^2}^2  \\
& \quad + C \frac{1}{N^2} \| (1-P_{\frac{N}{2^{\delta}} \le \cdot \le 2^{\delta}N})  \partial_x ( \sigma _{<\frac{1}{3} \frac{t}{N^2}} \ue_{N} ) (t) \|_{L^2}^2 \\
& \le \frac{(1+\delta) 2^{4\delta}}{3} \| \sigma _{<\frac{1}{3} \frac{t}{N^2}} \ue_{N} (t) \|_{L^2}^2 + C t^{-2} N^2 \| u_N (t) \|_{L^2}^2 .
\end{align*}
From the calculation used in \eqref{eq:Jd^2}, we have
\[
\| J \partial_x^2 (\sigma _{<\frac{1}{3} \frac{t}{N^2}} \ue_{N}) (t) \|_{L^2}
\lesssim N ( \| u_N (t) \|_{L^2} + \| J \partial_x u_N (t) \|_{L^2}) .
\]
Taking $f= \sigma _{<\frac{1}{3} \frac{t}{N^2}} \ue_{N}$ in \eqref{eq:ueob}, and by $2 \frac{(1+\delta) 2^{4\delta}}{3} <1$, we have
\[
\| \sigma _{<\frac{1}{3} \frac{t}{N^2}} \ue_{N} (t) \|_{L^2}
\lesssim t^{-1} N ( \| u_N (t) \|_{L^2} + \| J \partial_x u_N (t) \|_{L^2} ) .
\]

Because
\[
-\int_{\R} \frac{x}{t} | \sigma _{\frac{1}{3} \frac{t}{N^2} \le \cdot \le 3 \frac{t}{N^2}} (x) \ue_{N} (t,x)|^2 dx <0,
\]
and applying the calculation used in \eqref{eq:Jd^2} yields
\[
\| J \partial_x^2 (\sigma _{\frac{1}{3} \frac{t}{N^2} \le \cdot \le 3 \frac{t}{N^2}} \ue_{N}) (t) \|_{L^2}
\lesssim N ( \| u_N (t) \|_{L^2} + \| J \partial_x u_N (t) \|_{L^2}) ,
\]
taking $f=\sigma _{\frac{1}{3} \frac{t}{N^2} \le \cdot \le 3 \frac{t}{N^2}} \ue_{N}$ in \eqref{eq:ueob} gives
\[
\| \sigma _{\frac{1}{3} \frac{t}{N^2} \le \cdot \le 3 \frac{t}{N^2}} \ue_{N} (t) \|_{L^2}
\lesssim t^{-1} N ( \| u_N (t) \|_{L^2} + \| J \partial_x u_N (t) \|_{L^2} ) .
\]
\end{proof}

Set $\phi (t,x) := -2 \sqrt{t |x|}$.
The Gagliardo-Nirenberg inequality
\[
|f| \lesssim \| f \|_{L^2}^{\frac{1}{2}} \| \partial_x f \|_{L^2}^{\frac{1}{2}}
\]
with $f = e^{-i\phi} \uhp_{N}$, $\partial_x (e^{-i\phi} \uhp) = e^{-i\phi} \frac{1}{\sqrt{|x|}} J_+ \partial_x \uhp$, and Lemma \ref{lem:he_freq_est} imply
\begin{align*}
| \uhp_{N} (t,x) |
& \lesssim \| \uhp_{N} (t) \|_{L^2}^{\frac{1}{2}} \left\| \frac{1}{\sqrt{|x|}} J_+ \partial_x \uhp_{N} (t) \right\|_{L^2}^{\frac{1}{2}} \\
& \sim t^{-\frac{1}{4}} N^{\frac{1}{2}} \| u_{N} (t) \|_{L^2}^{\frac{1}{2}} \| J_+ \partial_x \uhp_{N} (t) \|_{L^2}^{\frac{1}{2}} \\
& \lesssim t^{-\frac{1}{2}} N \| u_{N} (t) \|_{L^2}^{\frac{1}{2}} (\| u_{N} (t) \|_{L^2} + \| J \partial_x u_{N} (t) \| )^{\frac{1}{2}} \\
& \lesssim t^{-\frac{1}{2}} \min (N^{-\frac{s}{2}+1}, N^{\frac{3}{2}}) \| u (t) \|_{X^s}.
\end{align*}
By
\[
J_+ \partial_x^2 \uhp_N (t,x)
= \partial_x J_+ \partial_x \uhp_N (t,x) + \frac{1}{2\sqrt{|x|}} \partial_x \uhp_N (t,x),
\]
Lemmas \ref{lem:freq_spat_loc} and \ref{lem:he_freq_est} imply
\begin{align*}
& \| J_+ \partial_x^2 \uhp_N (t) \|_{L^2} \\
& \lesssim \| P_{\frac{N}{2^{\delta}} \le \cdot \le 2^{\delta} N} \partial_x J_+ \partial_x \uhp_N (t) \|_{L^2} + \| (1-P_{\frac{N}{2^{\delta}} \le \le 2^{\delta} N}) \partial_x J_+ \partial_x \uhp_N (t) \|_{L^2} \\
& \quad + t^{-\frac{1}{2}} N^2 \| u_N (t) \|_{L^2} \\
& \lesssim N \| J_+ \partial_x \uhp_N (t) \|_{L^2} + t^{-\frac{1}{2}} N^2 \| u_N (t) \|_{L^2} \\
& \lesssim t^{-\frac{1}{2}} N^2 \| u_N (t) \|_{X^0} .
\end{align*}
Hence, we have
\begin{align*}
| \partial_x \uhp_{N} (t,x) |
\lesssim t^{-\frac{1}{2}} N^{2} \| u_N (t) \|_{L^2}^{\frac{1}{2}} \| u (t) \|_{X^0}^{\frac{1}{2}}
\lesssim t^{-\frac{1}{2}} \min(N^{-\frac{s}{2}+2}, N^{\frac{5}{2}}) \| u (t) \|_{X^s} .
\end{align*}
Because $\uhp (t,x)$ is a finite sum of $\uhp_N(t,x)$, we obtain the desired hyperbolic bounds.

Next, we show the elliptic bounds.
For $|x|\le \frac{t}{N^2}$, the Gagliardo-Nirenberg inequality and Lemmas \ref{lem:freq_spat_loc} and \ref{lem:he_freq_est} yield
\begin{align*}
|\ue_N(t,x)|
& = |\sigma_{\le \frac{t}{N^2}} \ue_N (t,x)| \\
& \le | P_{\frac{N}{2^{\delta}} \le \cdot \le 2^{\delta}N} \sigma_{\le \frac{t}{N^2}} \ue_N (t,x)| + | (1- P_{\frac{N}{2^{\delta}} \le \cdot \le 2^{\delta}N}) \sigma_{\le \frac{t}{N^2}} \ue_N (t,x)| \\
& \lesssim N^{\frac{1}{2}} \| \sigma_{\le \frac{t}{N^2}} \ue_N (t) \|_{L^2} + t^{-1} N^{\frac{3}{2}} \| u_N \|_{L^2} \\
& \lesssim t^{-\theta} \| |\partial_x|^{\frac{1+2\theta}{2-2\theta}} u_N (t) \|_{L^2}^{1-\theta} \| u_{N} (t) \|_{X^0}^{\theta} + t^{-1} N^{\frac{3}{2}} \| u_N \|_{L^2} ,
\end{align*}
where $0< \theta <1$.
For $|x|\ge \frac{t}{N^2}$, there exists $M \in 2^{\mathbb{N} \cup \{ 0 \}}$ such that $\ue_N (t,x) = \sigma_{\frac{t}{N^2}M}(x) \ue_N (t,x)$.
The calculation used above leads
\[
|\ue_N(t,x)|
\lesssim t^{-\theta} \Lr{N^2 \frac{x}{t}}^{-\theta} \| |\partial_x|^{\frac{1+2\theta}{2-2\theta}} u_N (t) \|_{L^2}^{1-\theta} \| u_{N} (t) \|_{X^0}^{\theta} + t^{-1} N^{\frac{3}{2}} \| u_N \|_{L^2} .
\]
Because
\[
\sum _{\substack{N \in 2^{\delta \mathbb{Z}} \\ N \le t}} \Lr{N^2 \frac{|x|}{t}}^{-\theta}
\le \sum _{\substack{N \in 2^{\delta \mathbb{Z}} \\ N \le t}} \
\lesssim 1+\log t,
\]
by setting $\theta = \frac{2s-1}{2s+2}$, we obtain
\[
| \ue (t,x) |
\le \sum_{\substack{N \in 2^{\delta \mathbb{Z}} \\ N \le t}} | \ue_{N} (t,x) |
\lesssim t^{-\frac{2s-1}{2s+2}} \left( 1+\log t \right) \| u (t) \|_{X^s}.
\]

From
\[
\sigma_{\le \frac{t}{N^2}} (x) \ue_x (t,x) = \partial_x (\sigma_{\le \frac{t}{N^2}} \ue) (t,x) - \frac{N^2}{t} \sigma' \left( N^2 \frac{x}{t} \right) \ue (t,x) ,
\]
the calculation used for $\ue_N$ yields, for $0 < \theta <1$,
\[
|\partial_x \ue_N (t,x)|
\lesssim 
t^{-\theta} \| |\partial_x|^{\frac{3+2\theta}{2-2\theta}} \ue_N (t) \|_{L^2}^{1-\theta} \| \ue_{N} (t) \|_{X^0}^{\theta} + t^{-1} N^{\frac{5}{2}} \| u_N \|_{L^2}
\]
if $|x| \le \frac{t}{N^2}$ and
\[
|\partial_x \ue_N (t,x)|
\lesssim 
t^{-\theta} \Lr{N^2 \frac{|x|}{t}}^{-\theta} \| |\partial_x|^{\frac{3+2\theta}{2-2\theta}} \ue_N (t) \|_{L^2}^{1-\theta} \| \ue_{N} (t) \|_{X^0}^{\theta} + t^{-1} N^{\frac{5}{2}} \| u_N \|_{L^2}
\]
if $|x| \ge \frac{t}{N^2}$.
Hence, setting $\theta = \frac{2s-3}{2s+2}$ gives
\[
| \ue_x (t,x) |
\lesssim \sum_{\substack{N \in 2^{\delta \mathbb{Z}} \\ N \le t}} | \partial_x \ue_{N} (t,x) |
\lesssim t^{-\frac{2s-3}{2s+2}} \left( 1+\log t \right) \| u (t) \|_{X^s}.
\]
\end{proof}

\begin{cor} \label{cor:uh}
For $s \ge 0$ and $t \ge 1$, we have
\[
\| \sqrt{|x|} J_+ \partial_x \uhp (t) \| _{L^2} \lesssim \| u (t) \|_{X^s} , \quad
\| x J_+ \partial_x \uhp_x (t) \|_{L^2} \lesssim t^{\frac{1}{2}} \| u (t) \|_{X^s} .
\]
\end{cor}

\begin{proof}
By Lemmas \ref{lem:freq_spat_loc} and \ref{lem:he_freq_est},
\begin{align*}
& \| \sqrt{|x|} J_+ \partial_x \uhp \| _{L^2} \\
& \lesssim \bigg( \sum _{\substack{N \in 2^{\delta \mathbb{Z}} \\ N \le t}} \| \sqrt{|x|} J_+ \partial_x \uhp_N (t) \| _{L^2}^2 \bigg)^{\frac{1}{2}} \\
& \quad + \sum _{\substack{N \in 2^{\delta \mathbb{Z}} \\ N \le t}} \| (1-P_{\frac{N}{2^{\delta}} \le \cdot \le 2^{\delta}N}) \sqrt{|x|} J_+ \partial_x \uhp_N (t) \| _{L^2} \\
& \lesssim \bigg( \sum _{\substack{N \in 2^{\delta \mathbb{Z}} \\ N \le t}} tN^{-2} \| J_+ \partial_x \uhp_N (t) \| _{L^2}^2 \bigg)^{\frac{1}{2}} + \sum _{\substack{N \in 2^{\delta \mathbb{Z}} \\ N \le t}} t^{-1} N \| u_N (t) \|_{L^2} \\
& \lesssim \| u (t) \|_{X^s}.
\end{align*}
Similarly, from $|x| J_+ \partial_x^2 \uhp_N = \partial_x (|x| J_+ \partial_x\uhp_N) + J_+ \partial_x\uhp_N + \frac{\sqrt{|x|}}{2} \partial_x \uhp_N$, we have
\begin{align*}
& \| x J_+ \partial_x \uhp _x \| _{L^2} \\
& \lesssim \bigg( \sum _{\substack{N \in 2^{\delta \mathbb{Z}} \\ N \le t}} (t N^{-1}+1)^2 \| J_+ \partial_x \uhp_N (t) \| _{L^2}^2 \bigg)^{\frac{1}{2}} + t^{\frac{1}{2}} \bigg( \sum _{\substack{N \in 2^{\delta \mathbb{Z}} \\ N \le t}} \| \uhp_N (t) \|_{L^2}^2 \bigg) ^{\frac{1}{2}} \\
& \qquad +\sum _{\substack{N \in 2^{\delta \mathbb{Z}} \\ N \le t}} \| (1-P_{\frac{N}{2^{\delta}} \le \cdot \le 2^{\delta}N}) |x| J_+ \partial_x^2 \uhp_N (t) \| _{L^2} \\
& \lesssim t^{\frac{1}{2}} \| u (t) \|_{X^s}.
\end{align*}
\end{proof}

\section{Wave packets} \label{S:wave_packet}

We consider the Hamiltonian flow corresponding to \eqref{sp}, which is given by
\[
(x,\xi) \mapsto \Big( x-\frac{t}{\xi^2},\xi \Big) .
\]
We expect solutions initially localized spatially near zero and in frequency near $\pm \xi_v$, where $\xi_v := \frac{1}{\sqrt{|v|}}$, to travel along the ray $\Gamma _{v} := \{ x=vt \}$ when $v<0$.
This produces a phase function
\[
\phi (t,x) := -2 \sqrt{t|x|}
\]
associated with the linear propagator $e^{t \partial_x^{-1}}$.

For $v \in \R_-$, we define
\[
\Psi _v(t,x) := |v|^{-\frac{3}{4}} \chi \left( \frac{x-vt}{t^{\frac{1}{2}} |v|^{\frac{3}{4}}} \right) e^{i \phi (t,x)},
\]
where $\chi$ is a sooth function with $\supp \chi \subset [-1+2^{-\delta}, 1-2^{-\delta}]$ and $\int_{\R} \chi dx =1$.
The spatial support of $\Psi_v$ is included in $[2^{\delta}vt, \frac{vt}{2^{\delta}}]$, provided that $v \in \R_-$ and $|v| \ge t^{-2}$.

Let
\[
\Omega_{\alpha} (t) := \{ v \in \R_- : t^{-\alpha} \le -v \le t^{\alpha} \}
\]
for $t \ge 1$ and $\alpha >0$.
For $v \in \Omega_2 (t)$,
\[
\partial_t \Psi _v (t,x)
= - \frac{x+vt}{2 t^{\frac{3}{2}} |v|^{\frac{3}{2}}} \chi' \left( \frac{x-vt}{t^{\frac{1}{2}} |v|^{\frac{3}{4}}} \right) e^{i\phi (t,x)} - i t^{-\frac{1}{2}} |v|^{-\frac{3}{4}} \sqrt{|x|} \chi \left( \frac{x-vt}{t^{\frac{1}{2}} |v|^{\frac{3}{4}}} \right) e^{i\phi (t,x)} .
\]
Integrating by parts three times gives
\begin{align*}
& \partial_x^{-1} \Psi_v (t,x) \\
&=-i t^{-\frac{1}{2}} |v|^{-\frac{3}{4}} \sqrt{|x|} \chi \left( \frac{x-vt}{t^{\frac{1}{2}} |v|^{\frac{3}{4}}} \right) e^{i\phi (t,x)} -\frac{1}{2} t^{-1} |v|^{-\frac{3}{4}} \chi \left( \frac{x-vt}{t^{\frac{1}{2}} |v|^{\frac{3}{4}}} \right) e^{i\phi (t,x)} \\
& \quad + t^{-\frac{3}{2}} |v|^{-\frac{3}{2}} |x| \chi' \left( \frac{x-vt}{t^{\frac{1}{2}} |v|^{\frac{3}{4}}} \right) e^{i\phi (t,x)} + i t^{-2} |v|^{-\frac{3}{2}} \partial_x \left\{ |x|^{\frac{3}{2}} \chi' \left( \frac{x-vt}{t^{\frac{1}{2}} |v|^{\frac{3}{4}}} \right) \right\} e^{i\phi (t,x)} \\
& \quad - i t^{-2} |v|^{-\frac{3}{2}} \partial_x^{-1} \left( \partial_x ^2 \left\{ |x|^{\frac{3}{2}} \chi' \left( \frac{x-vt}{t^{\frac{1}{2}} |v|^{\frac{3}{4}}} \right) \right\} e^{i\phi (t,x)} \right) .
\end{align*}
Hence, we have
\begin{equation} \label{eq:LPsi}
\begin{aligned}
& (L \Psi_v) (t,x) \\
& = t^{-1} \left\{ \frac{1}{2 |v|^{\frac{3}{4}}} \chi \left( \frac{x-vt}{t^{\frac{1}{2}} |v|^{\frac{3}{4}}} \right) + \frac{x-vt}{2 t^{\frac{1}{2}} |v|^{\frac{3}{2}}} \chi' \left( \frac{x-vt}{t^{\frac{1}{2}} |v|^{\frac{3}{4}}} \right) \right\} e^{i\phi (t,x)} \\
& \quad - i t^{-2} |v|^{-\frac{3}{2}} \partial_x \left\{ |x|^{\frac{3}{2}} \chi' \left( \frac{x-vt}{t^{\frac{1}{2}} |v|^{\frac{3}{4}}} \right) \right\} e^{i\phi (t,x)} \\
& \quad + i t^{-2} |v|^{-\frac{3}{2}} \partial_x^{-1} \left( \partial_x ^2 \left\{ |x|^{\frac{3}{2}} \chi' \left( \frac{x-vt}{t^{\frac{1}{2}} |v|^{\frac{3}{4}}} \right) \right\} e^{i\phi (t,x)} \right) \\
& = t^{-1} (\partial_x \wt{\chi}) (t,x) e^{i\phi (t,x)} + i t^{-2} |v|^{-\frac{3}{2}} \partial_x^{-1} \left( \partial_x ^2 \left\{ |x|^{\frac{3}{2}} \chi' \left( \frac{x-vt}{t^{\frac{1}{2}} |v|^{\frac{3}{4}}} \right) \right\} e^{i\phi (t,x)} \right) 
\end{aligned}
\end{equation}
where
\[
\wt{\chi}(t,x) := \frac{x-vt}{2|v|^{\frac{3}{4}}} \chi \left( \frac{x-vt}{t^{\frac{1}{2}} |v|^{\frac{3}{4}}} \right) -i \frac{|x|^{\frac{3}{2}}}{t |v|^{\frac{3}{2}}} \chi' \left( \frac{x-vt}{t^{\frac{1}{2}} |v|^{\frac{3}{4}}} \right) .
\]

We show that $\Psi_v(t,x)$ and the first part of $L \Psi_v (t,x)$ are essentially frequency localized near $\xi_v$.
To state this more precisely, for $v \in \Omega_2(t)$ we define by $N_{v} \in 2^{\delta \mathbb{Z}}$ the nearest scaled dyadic number to $\xi_v$.
Then, $\frac{\xi_v}{2^{\delta}} < N_v < 2^{\delta} \xi_v$ holds.

\begin{lem} \label{lem:freq_psi}
For $t \ge 1$ and $v \in \Omega_2(t)$, we have
\begin{align*}
& \| (1-P_{\frac{N_v}{2^{\delta}} \le \cdot \le 2^{\delta}N_v}^+) \Psi_v(t) \|_{L^2}
\lesssim_c t^{\frac{1}{4}} |v|^{-\frac{3}{8}} ( t^{\frac{1}{2}} |v|^{\frac{1}{4}} )^{-c} , \\
& \| (1-P_{\frac{N_v}{2^{\delta}} \le \cdot \le 2^{\delta}N_v}^+) e^{i\phi (t)} (\partial_x \wt{\chi}) (t) \|_{L^2}
\lesssim_c t^{\frac{1}{4}} |v|^{-\frac{3}{8}} ( t^{\frac{1}{2}} |v|^{\frac{1}{4}} )^{-c}
\end{align*}
for any $c \ge 0$.
\end{lem}

\begin{proof}
From Taylor's theorem, we can write
\begin{align*}
\phi (t,x)
& = \phi (t,vt) + \partial_x \phi (t,vt) (x-vt) + \frac{1}{2} \partial_x^2 \phi (t,vt) (x-vt)^2 \\
& \qquad + \int_{vt}^x \frac{(x-y)^2}{2} \partial_x^3 \phi (t,y) dy \\
& = -2t \sqrt{|v|} + \frac{1}{\sqrt{|v|}} (x-vt) + \frac{1}{4 t |v|^{\frac{3}{2}}} (x-vt)^2 + R \left( \frac{x-vt}{t^{\frac{1}{2}} |v|^{\frac{3}{4}}}, t^{\frac{1}{2}} |v|^{\frac{1}{4}} \right) ,
\end{align*}
where
\[
R(x,a) := \frac{3}{8} \frac{x^3}{a} \int _0^1 \frac{(1-\theta)^2}{(-\theta \frac{x}{a}+1)^{\frac{5}{2}}} d\theta .
\]
We note that $R (x,a)$ is well-defined provided that $\max (x,0) < a$.
Changing the variable $y= \frac{x-vt}{t^{\frac{1}{2}} |v|^{\frac{3}{4}}}$, we have
\begin{align*}
& \mathcal{F} [\Psi_v] (t,\xi) \\
& = \frac{1}{\sqrt{2\pi}} \int_{\R} e^{-ix\xi} |v|^{-\frac{3}{4}} \chi \left( \frac{x-vt}{t^{\frac{1}{2}} |v|^{\frac{3}{4}}} \right) e^{i \phi (t,x)} dx \\
& = \frac{1}{\sqrt{2\pi}} t^{\frac{1}{2}} e^{i (-2t\sqrt{|v|} + t |v| \xi)} \int _{\R} e^{-i y t^{\frac{1}{2}} |v|^{\frac{3}{4}} (\xi - \xi_v)} \chi (y) e^{\frac{i}{4} y^2 + i R(y, t^{\frac{1}{2}} |v|^{\frac{1}{4}})} dy  \\
& = t^{\frac{1}{2}} \chi_1 \left( t^{\frac{1}{2}} |v|^{\frac{3}{4}} (\xi - \xi_v), t^{\frac{1}{2}} |v|^{\frac{1}{4}} \right) ,
\end{align*}
where
\[
\chi_1 (\xi, a) := e^{i(a\xi-a^2)} \mathcal{F} [e^{\frac{i}{4} x^2 + i R(x, a)} \chi ] (\xi) .
\]
By definition, $\chi_1 (\cdot ,a ) \in \mathcal{S}(\R)$ for $a \ge 1$.
From
\[
\| (\cdot - \xi_v)^c t^{\frac{1}{2}} \chi_1 \left( t^{\frac{1}{2}} |v|^{\frac{3}{4}} (\cdot - \xi_v), t^{\frac{1}{2}} |v|^{\frac{1}{4}} \right) \|_{L^2_{\xi}}
\lesssim_c t^{\frac{1}{4}} |v|^{-\frac{3}{8}} (t^{\frac{1}{2}} |v|^{\frac{3}{4}})^{-c} \| \chi \|_{L^2_x}
\]
and $|\xi - \xi_v| \ge (1-2^{-\delta}) \xi_v$, provided that $\xi \notin [ \frac{N_v}{2^{2\delta}}, 2^{2\delta} N_v]$, we obtain the $L^2$ bound.

Next, we focus on the estimate for $\partial_x \wt{\chi}$.
Setting
\[
\wt{\chi}_0 (x,a) := \frac{x}{2} \chi (x) - i a^{-\frac{3}{2}} |x-a|^{\frac{3}{2}} \chi' (x),
\]
we can write
\[
\wt{\chi}(t,x) = t^{\frac{1}{2}} \wt{\chi}_0 \left( \frac{x-vt}{t^{\frac{1}{2}} |v|^{\frac{3}{4}}}, t^{\frac{1}{2}} |v|^{\frac{1}{4}} \right) .
\]
Here, $\wt{\chi}_0 ( \cdot , a) \in \mathcal{S}(\R)$ for $a \ge 1$.
The calculation used for $\mathcal{F}[ \Psi_v]$ yields
\[
\mathcal{F}[e^{i\phi} (\partial_x \wt{\chi})](t,\xi)
= t^{\frac{1}{2}} \wt{\chi}_1 \left( t^{\frac{1}{2}} |v|^{\frac{3}{4}} (\xi - \xi_v), t^{\frac{1}{2}} |v|^{\frac{1}{4}} \right) ,
\]
where
\[
\wt{\chi}_1 (\xi, a) := e^{i(a\xi-a^2)} \mathcal{F} [e^{\frac{i}{4} x^2 + i R(x, a)} \partial_x \wt{\chi}_0 ( \cdot , a) ] (\xi) .
\]
This gives the desired bound, as above.
\end{proof}

For $v \in \R_-$, we define
\[
\gamma (t,v) := \int_{\R} u (t,x) \overline{\Psi}_v (t,x) dx.
\]
Because $\ue$ and $\uhn$ are essentially frequency localized away from $\xi_v$, we can replace $u$ on the right hand side with $\uhp$.
Indeed, by applying H\"{o}lder's inequality, Lemmas \ref{lem:freq_spat_loc} and \ref{lem:freq_psi}, and Proposition \ref{prop:est|u|}, we obtain
\begin{equation} \label{gamma_uhp}
\begin{aligned}
& \left| \gamma (t,v) - \int_{\R} \uhp (t,x) \overline{\Psi}_v(t,x) dx \right| \\
& \le \| u (t) \|_{L^2} \| (1-P_{\frac{N_v}{2^{\delta}} \le \cdot \le 2^{\delta}N_v}^+) \Psi_v(t) \|_{L^2} + \| \ue (t) \overline{\Psi_v} (t) \|_{L^1} \\
& \quad + \| P_{\frac{N_v}{2^{\delta}} \le \cdot \le 2^{\delta}N_v}^+ \overline{\uhp_N} (t) \|_{L^{\infty}} \| \Psi_{v}(t) \|_{L^1} \\
& \lesssim t^{-\frac{1}{4}} |v|^{-\frac{5}{8}} \| u (t) \|_{L^2} + t^{-\frac{2s-1}{2s+2}+\frac{1}{2}} \left( 1+ \log t \right) \| u (t) \|_{X^s}
+ t^{-\frac{1}{2}} \sum _{\substack{N \in 2^{\delta \mathbb{Z}} \\ N \sim N_v}} N^{\frac{3}{2}} \| u_N (t) \|_{L^2} \\
& \lesssim \left( t^{-\frac{1}{4}} |v|^{-\frac{5}{8}} + t^{-\frac{2s-1}{2s+2}+\frac{1}{2}} \left( 1+ \log t \right) \right) \| u (t) \|_{X^s} ,
\end{aligned}
\end{equation}
provided that $s>\frac{5}{2}$, $t \ge 1$, and $v \in \Omega_2(t)$.

\begin{prop} \label{prop:approx}
For $t \ge 1$,
\[
\| \gamma (t) \|_{L^{\infty}_{v} (\R_-)} \lesssim t^{\frac{1}{2}} \| u (t) \|_{L^{\infty}_x} .
\]
Let $s > \frac{5}{2}$ and $0< \alpha < \frac{2}{2s-1}$.
For $t \ge 1$ and $v \in \Omega_{\alpha} (t)$, we have the bounds
\begin{align*}
& | u(t,vt) - 2 t^{-\frac{1}{2}} \Re \{ e^{i\phi (t,vt)} \gamma (t,v) \} | \lesssim \left\{ t^{-\frac{3}{4}} |v|^{-\frac{5}{8}} + t^{-\frac{2s-1}{2s+2}} (1+ \log t) \right\} \| u (t) \|_{X^s} , \\
&
\begin{aligned}
& | u_x (t,vt) - 2 t^{-\frac{1}{2}} |v|^{-\frac{1}{2}} \Re \{ i e^{i\phi (t,vt)} \gamma (t,v) \} | \\
& \hspace*{140pt} \lesssim \left\{ t^{-\frac{3}{4}} |v|^{-\frac{9}{8}} + t^{-\frac{2s-3}{2s+2}} (1+ \log t) \right\} \| u (t) \|_{X^s} .
\end{aligned}
\end{align*}
\end{prop}

\begin{proof}
By direct calculation,
\[
\| \gamma (t) \|_{L^{\infty}_v}
\le t^{\frac{1}{2}} \| u (t) \|_{L^{\infty}_x} \int_{\R} | \chi (x) | dx
\sim t^{\frac{1}{2}} \| u (t) \|_{L^{\infty}_x}.
\]

We set $\whp (t,x) := e^{-i\phi (t,x)} \uhp (t,x)$.
As $u = 2 \Re u^+$, from \eqref{gamma_uhp} and Proposition \ref{prop:est|u|}, we have
\begin{align*}
& u(t,vt) - 2 t^{-\frac{1}{2}} \Re \{ e^{i\phi (t,vt)} \gamma (t,v) \} \\
& = 2 \Re \left[ e^{i\phi (t,vt)} \left\{ \whp(t,vt) - t^{-\frac{1}{2}} |v|^{-\frac{3}{4}} \int_{\R} \whp(t,x) \chi \left( \frac{x-vt}{t^{\frac{1}{2}} |v|^{\frac{3}{4}}} \right) dx \right\} \right] \\
& \qquad + O \left( \left( t^{-\frac{3}{4}} |v|^{-\frac{5}{8}} + t^{-\frac{2s-1}{2s+2}} (1+ \log t) \right) \| u (t) \|_{X^s} \right) .
\end{align*}
Here by changing the variable $z = \frac{x-vt}{t}$,
\begin{align*}
& \left| \whp(t,vt) - t^{-\frac{1}{2}} |v|^{-\frac{3}{4}} \int_{\R} \whp(t,x) \chi \left( \frac{x-vt}{t^{\frac{1}{2}} |v|^{\frac{3}{4}}} \right) dx \right| \\
& = t^{-\frac{1}{2}} |v|^{-\frac{3}{4}} \left| \int_{\R} \left\{ \whp(t,vt) - \whp(t,x) \right\} \chi \left( \frac{x-vt}{t^{\frac{1}{2}} |v|^{\frac{3}{4}}} \right) dx \right| \\
& = t^{\frac{1}{2}} |v|^{-\frac{3}{4}} \left| \int_{\R} \left\{ \whp(t,vt) - \whp(t, t(z+v)) \right\} \chi (t^{\frac{1}{2}} |v|^{-\frac{3}{4}} z) dz \right| .
\end{align*}
From $|z| \le (1-2^{-\delta}) t^{-\frac{1}{2}} |v|^{\frac{3}{4}}$, $v \in \Omega_2 (t)$, and Corollary \ref{cor:uh},
\begin{equation} \label{eq:diffw}
\begin{aligned}
& |\whp (t,vt) - \whp (t, t(z+v))| \\
& = |tz| \left| \int_0^1 \partial_x \whp (t, vt+(1-\theta)tz) d\theta \right| \\
& \lesssim |tz|^{\frac{1}{2}} \| \partial_x \whp (t) \|_{L^2([\frac{t|v|}{2^{\delta}}, 2^{\delta}t|v|])} \\
& \sim |z|^{\frac{1}{2}} t^{-\frac{1}{2}} |v|^{-1} \left\| \sqrt{|x|} J_+ \partial_x \uhp (t) \right\|_{L^2([\frac{t|v|}{2^{\delta}}, 2^{\delta}t|v|])} \\
& \lesssim |z|^{\frac{1}{2}} t^{-\frac{1}{2}} |v|^{-1} \| u(t) \|_{X^s} .
\end{aligned}
\end{equation}
Therefore, we obtain
\begin{align*}
& \left| \whp (t,vt) - t^{-\frac{1}{2}} |v|^{-\frac{3}{4}} \int_{\R} \whp (t,x) \chi \left( \frac{x-vt}{t^{\frac{1}{2}} |v|^{\frac{3}{4}}} \right) dx \right| \\
& \lesssim |v|^{-\frac{7}{4}} \int_{\R} |z|^{\frac{1}{2}} |\chi (t^{\frac{1}{2}} |v|^{-\frac{3}{4}} z)| dz \| u (t) \|_{X^s} \\
& \lesssim t^{-\frac{3}{4}} |v|^{-\frac{5}{8}} \| u (t) \|_{X^s} .
\end{align*}

Next, we show the approximation estimate of $u_x$.
By applying integration by parts and Proposition \ref{prop:est|u|}, we have
\begin{align*}
& \int_{\R} \uhp (t,x) \overline{\Psi}_v(t,x) dx \\
&= -i t^{-\frac{1}{2}} |v|^{-\frac{3}{4}} \int_{\R} \sqrt{|x|} \; \uhp_x (t,x) \chi \left( \frac{x-vt}{t^{\frac{1}{2}} |v|^{\frac{3}{4}}} \right) e^{-i \phi (t,x)} dx \\
& \quad -i t^{-1} |v|^{-\frac{3}{2}} \int_{\R} \sqrt{|x|} \; \uhp (t,x) \chi' \left( \frac{x-vt}{t^{\frac{1}{2}} |v|^{\frac{3}{4}}} \right) e^{-i \phi (t,x)} dx \\
& \quad + \frac{1}{2} i t^{-\frac{1}{2}} |v|^{-\frac{3}{4}} \int_{\R} \frac{1}{\sqrt{|x|}} \uhp (t,x) \chi \left( \frac{x-vt}{t^{\frac{1}{2}} |v|^{\frac{3}{4}}} \right) e^{-i \phi (t,x)} dx \\
& = -i |v|^{-\frac{1}{4}} \int_{\R} \uhp_x (t,x) \chi \left( \frac{x-vt}{t^{\frac{1}{2}} |v|^{\frac{3}{4}}} \right) e^{-i \phi (t,x)} dx \\
& \quad + O \left( t^{-\frac{1}{2}} |v|^{\frac{1}{4}} \min ( |v|^{\frac{s}{4}-1}, |v|^{-\frac{5}{4}}) \| u (t) \|_{X^s} \right) .
\end{align*}
Hence, from \eqref{gamma_uhp}, $0< \alpha < \min ( \frac{2}{3}, \frac{2}{2s-1})$, and Proposition \ref{prop:est|u|}, we can write
\begin{align*}
& u_x (t,vt) - 2 t^{-\frac{1}{2}} |v|^{-\frac{1}{2}} \Re \{ i e^{i\phi (t,vt)} \gamma (t,v) \} \\
& = 2 \Re \left\{ \uhp_x (t,vt) - t^{-\frac{1}{2}} |v|^{-\frac{3}{4}} e^{i\phi (t,vt)} \int_{\R} \uhp_x (t,x) \chi \left( \frac{x-vt}{t^{\frac{1}{2}} |v|^{\frac{3}{4}}} \right) e^{-i\phi (t,x)} dx \right\} \\
& \quad + O \left( \left( t^{-\frac{3}{4}} |v|^{-\frac{9}{8}} + t^{-\frac{2s-3}{2s+2}} (1+ \log t) \right) \| u (t) \|_{X^s} \right) .
\end{align*}
For $|z| \le (1-2^{-\delta}) t^{-\frac{1}{2}} |v|^{\frac{3}{4}}$, $v \in \Omega_2(t)$, and by Corollary \ref{cor:uh},
\begin{align*}
\| |x|^{-\frac{1}{2}} J_+ \partial_x \uhp_x (t) \|_{L^2([\frac{t|v|}{2^{\delta}}, 2^{\delta}t|v|])}
& \sim (t|v|)^{-\frac{3}{2}} \| |x| J_+ \partial_x \uhp_x (t) \|_{L^2([\frac{t|v|}{2^{\delta}}, 2^{\delta}t|v|])} \\
& \lesssim t^{-1} |v|^{-\frac{3}{2}} \| u(t) \|_{X^s}.
\end{align*}
By the argument given above, we obtain the desired bound.
\end{proof}

\begin{prop} \label{prop:gamma_decay}
Let $s>4$ and $0< \alpha < \min \left\{ \frac{2}{45}, \frac{2}{s+1}, \frac{2(s-4)}{3(s+1)} \right\}$.
If $u$ solves \eqref{sp}, then, for $t \ge 1$ and $v \in \Omega_{\alpha} (t)$, we have
\[
\dot{\gamma}(t,v) = 3it^{-1} |v|^{-\frac{1}{2}} |\gamma (t,v)|^2 \gamma (t,v) + O\left( t^{-\frac{6}{5}} (\| u (t) \|_{X^s} + \| u (t) \|_{X^s}^3) \right) .
\]
\end{prop}

\begin{proof}
By \eqref{eq:LPsi},
\begin{align*}
\dot{\gamma}(t,v)
& = \int_{\R} ( Lu \cdot \overline{\Psi}_v + u L \overline{\Psi}_v ) (t,x) dx \\
& = \int_{\R} \partial_x(u^3) (t,x) \overline{\Psi}_v (t,x) dx + t^{-1} \int_{\R} e^{-i\phi (t,x)} u (t,x) \partial_x \overline{\wt{\chi}} (t,x) dx \\
& \qquad + i t^{-2} |v|^{-\frac{3}{2}} \int_{\R} e^{-i\phi (t,x)} \partial_x^{-1} u (t,x) \partial_x ^2 \left\{ |x|^{\frac{3}{2}} \chi' \left( \frac{x-vt}{t^{\frac{1}{2}} |v|^{\frac{3}{4}}} \right) \right\} dx .
\end{align*}
The calculation used in \eqref{gamma_uhp} yields
\begin{align*}
& \left| \int_{\R} e^{-i\phi (t,x)} u (t,x) \partial_x \overline{\wt{\chi}} (t,x) dx - \int_{\R} e^{-i\phi (t,x)} \uhp (t,x) \partial_x \overline{\wt{\chi}} (t,x) dx \right| \\
& \lesssim \| u_N (t) \|_{L^2} \| (1-P_{\frac{N_v}{2^{\delta}} \le \cdot \le 2^{\delta}N_v}^+) e^{i \phi} \partial_x \wt{x} \|_{L^2} + \| \ue (t) \partial_x \overline{\wt{x}} (t) \|_{L^1} \\
& \quad + \| P_{\frac{N}{2^{\delta}} \le \cdot \le 2^{\delta}N}^+ \overline{\uhp_N} (t) \|_{L^{\infty}} \| \partial_x \wt{x} (t) \|_{L^1} \\
& \lesssim t^{-\frac{3}{4}} |v|^{-\frac{7}{8}} \| u (t) \|_{L^2} + t^{-\frac{2s-1}{2s+2}+\frac{1}{2}} \left( 1+ \log t \right) \| u (t) \|_{X^s} + t^{-\frac{1}{2}} \sum _{\substack{N \in 2^{\delta \mathbb{Z}} \\ N \sim N_v}} N^{\frac{3}{2}} \| u_N (t) \|_{L^2} \\
& \lesssim t^{-\frac{1}{5}} \| u (t) \|_{X^s}
\end{align*}
provided that $s>4$ and $0< \alpha < \frac{2}{5}$.
From Corollary \ref{cor:uh},
\begin{align*}
\left| \int_{\R} \partial_x (e^{-i\phi (t,x)} \uhp (t,x)) \overline{\wt{\chi}} (t,x) dx \right|
& \lesssim t^{-1} |v|^{-1} \| \sqrt{|x|} J_+ \partial_x \uhp (t) \|_{L^2} \| \wt{\chi} (t) \|_{L^2} \\
& \lesssim t^{-\frac{1}{4}} |v|^{-\frac{5}{8}} \| u (t) \|_{X^s} .
\end{align*}
From
\begin{align*}
\partial_x ^2 \left\{ |x|^{\frac{3}{2}} \chi' \left( \frac{x-vt}{t^{\frac{1}{2}} |v|^{\frac{3}{4}}} \right) \right\}
& = \frac{3}{4} \frac{1}{\sqrt{|x|}} \chi' \left( \frac{x-vt}{t^{\frac{1}{2}} |v|^{\frac{3}{4}}} \right) - 3 \frac{\sqrt{|x|}}{t^{\frac{1}{2}} |v|^{\frac{3}{4}}} \chi'' \left( \frac{x-vt}{t^{\frac{1}{2}} |v|^{\frac{3}{4}}} \right) \\
& \quad + \frac{|x|^{\frac{3}{2}}}{t |v|^{\frac{3}{2}}} \chi''' \left( \frac{x-vt}{t^{\frac{1}{2}} |v|^{\frac{3}{4}}} \right)
\end{align*}
we have
\begin{align*}
& \left| \int_{\R} e^{-i\phi (t,x)} \partial_x^{-1} u (t,x) \partial_x ^2 \left\{ |x|^{\frac{3}{2}} \chi' \left( \frac{x-vt}{t^{\frac{1}{2}} |v|^{\frac{3}{4}}} \right) \right\} dx \right| \\
& \le \| \partial_x^{-1} u (t) \|_{L^2} \left\| \partial_x ^2 \left\{ |\cdot |^{\frac{3}{2}} \chi' \left( \frac{\cdot -vt}{t^{\frac{1}{2}} |v|^{\frac{3}{4}}} \right) \right\} \right\|_{L^2}
\lesssim t^{\frac{3}{4}} |v|^{\frac{3}{8}} \| u (t) \|_{X^s} .
\end{align*}
This yields
\begin{equation} \label{eq:dotgamma}
\dot{\gamma}(t,v) = \int_{\R} \partial_x (u^3) (t,x) \overline{\Psi}_v (t,x) dx + O ( t^{-\frac{6}{5}} \| u (t) \|_{X^s} )
\end{equation}
as $s>4$ and $0< \alpha < \frac{2}{45}$.

Because
\[
\partial_x \Psi_v (t,x) =t^{-\frac{1}{2}} |v|^{-\frac{3}{2}} e^{i\phi (t,x)} \chi' \left( \frac{x-vt}{t^{\frac{1}{2}} |v|^{\frac{3}{4}}} \right) + i \sqrt{\frac{t}{|x|}} \Psi_v (t,x)
\]
and Proposition \ref{prop:est|u|} and $0< \alpha < \min \left\{ \frac{2}{s+1}, \frac{2(s-2)}{3(s+1)} \right\}$ yield
\[
|u(t,vt)| \lesssim t^{-\frac{1}{2}} \min ( |v|^{\frac{s}{4}-\frac{1}{2}}, |v|^{-\frac{3}{4}} ) \| u (t) \|_{X^s},
\]
we have
\begin{align*}
& \int_{\R} \partial_x (u^3) (t,x) \overline{\Psi}_v (t,x) dx \\
& = i |v|^{-\frac{1}{2}} \int_{\R} u^3(t,x) \overline{\Psi}_v (t,x) dx + O \left( t^{-\frac{6}{5}} \| u (t) \|_{X^s}^3 \right) \\
& = i |v|^{-\frac{1}{2}} \int_{\R} (\uh)^3 (t,x) \overline{\Psi}_v(t,x) dx + O \left( t^{-\frac{6}{5}} \| u (t) \|_{X^s}^3 \right)
\end{align*}
provided that $s>4$ and $0< \alpha < \min \left\{ \frac{2}{s+1}, \frac{2(s-2)}{3(s+1)} \right\}$.
Here, we observe that for $\frac{|v|}{2^{\delta}} \le \frac{|x|}{t} \le 2^{\delta}|v|$,
\begin{align*}
\uhp
& = \sum_{\substack{N \in 2^{\delta \mathbb{Z}} \\ \frac{N_v}{\sqrt{3} 2^{2\delta}} \le N \le \sqrt{3} 2^{2\delta} N_v}} \uhp_N \\
& = P_{\frac{N_v}{\sqrt{3} 2^{3\delta}} \le \cdot \le \sqrt{3} 2^{3\delta} N_v}^+ \uhp + \sum_{\substack{N \in 2^{\delta \mathbb{Z}} \\ \frac{N_v}{\sqrt{3} 2^{2\delta}} \le N \le \sqrt{3} 2^{2\delta} N_v}} (1-P_{\frac{N}{2^{\delta}} \le \cdot \le 2^{\delta} N}^+) \uhp_N .
\end{align*}
If the frequency support of
\[
(\uh)^3 - 3 |\uhp|^2 \uhp = (\uhp)^3 + 3 | \uhp|^2 \overline{\uhp} + \overline{\uhp}^3,
\]
is contained in $[\frac{N_v}{2^{2\delta}}, 2^{2\delta} N_v]$, then at least one of $\uhp$ on the right hand side is $(1-P_{\frac{N}{2^{\delta}} \le \cdot \le 2^{\delta} N}^+) \uhp_N$.
Accordingly, for $s>4$ and $0< \alpha < \frac{2}{5}$, Lemmas \ref{lem:freq_spat_loc} and \ref{lem:freq_psi} and Proposition \ref{prop:est|u|} imply
\begin{align*}
& \left| \int_{\R} (\uh)^3 (t,x) \overline{\Psi}_v(t,x) dx - 3 \int_{\R} (|\uhp|^2 \uhp) (t,x) \overline{\Psi}_v(t,x) dx \right| \\
& \lesssim t^{-1} \min (|v|^{\frac{s}{2}-1}, |v|^{-\frac{3}{2}}) \bigg( \| \uhp (t) \|_{L^2} \| (1-P_{\frac{N_v}{2^{\delta}} \le \cdot \le 2^{\delta} N_v}) \Psi_v (t) \|_{L^2} \\
& \quad + \sum_{\substack{N \in 2^{\delta \mathbb{Z}} \\ N \sim N_v}} \| (1-P_{\frac{N}{2^{\delta}} \le \cdot \le 2^{\delta} N}^+) \uhp_N (t) \|_{L^{\infty}} \| \Psi _v (t) \|_{L^{1}} \bigg) \| u(t) \|_{X^s}^2 \\
& \lesssim t^{-1} \min (|v|^{\frac{s}{2}-1}, |v|^{-\frac{3}{2}}) \left( t^{-\frac{1}{4}} |v|^{-\frac{5}{8}} + t^{-\frac{1}{2}} |v|^{-\frac{3}{4}} \right) \| u (t) \|_{X^s}^3 \\
& \lesssim t^{-\frac{6}{5}} |v|^{\frac{1}{2}} \| u (t) \|_{X^s}^3 .
\end{align*}
Moreover, \eqref{eq:diffw} and $0 < \alpha < \frac{2}{15}$ yield
\begin{align*}
& \left| \int_{\R} \uhp (t,x) (|\uhp (t,x)|^2-|\uhp (t,vt)|^2) \overline{\Psi}_v (t,x) dx \right| \\
& \lesssim t^{-1} |v|^{-\frac{3}{4}} \min ( |v|^{\frac{s}{2}-1} , |v|^{-\frac{3}{2}}) \| u (t) \|_{X^s}^2 \\
& \quad \times \int_{\R} \left| \chi \left( \frac{x-vt}{t^{\frac{1}{2}} |v|^{\frac{3}{4}}} \right) \right| |\whp (t,x) - \whp (t,vt)| dx \\
& \lesssim \| u (t) \|_{X^s}^2 \int_{\R} \left| \chi (t^{\frac{1}{2}} |v|^{-\frac{3}{4}} z) \right| |\whp (t,t(z+v)) - \whp (t,vt)| dz \\
& \lesssim t^{-\frac{1}{2}} |v|^{-1} \| u (t) \|_{X^s}^3 \int_{\R} |z|^{\frac{1}{2}} \left| \chi (t^{\frac{1}{2}} |v|^{-\frac{3}{4}} z) \right| dz \\
& \lesssim t^{-\frac{5}{4}} |v|^{\frac{1}{8}} \| u (t) \|_{X^s}^3
\lesssim t^{-\frac{6}{5}} |v|^{\frac{1}{2}} \| u (t) \|_{X^s}^3 .
\end{align*}
Therefore, we have
\begin{align*}
& \int_{\R} \partial_x (u^3) (t,x) \overline{\Psi}_v (t,x) dx \\
& = 3 i |v|^{-\frac{1}{2}} |\uhp (t,vt)|^2 \gamma (t,v) \\
& \quad + 3i |v|^{-\frac{1}{2}} \int_{\R} \uhp (t,x) (|\uhp (t,x)|^2-|\uhp (t,vt)|^2) \overline{\Psi}_v (t,x) dx \\
& \quad + O \left( t^{-\frac{6}{5}} \| u (t) \|_{X^s}^3 \right) \\
& = 3it^{-1} |v|^{-\frac{1}{2}} |\gamma|^2 \gamma + O \left( t^{-\frac{6}{5}} \| u (t) \|_{X^s}^3 \right) .
\end{align*}
Combining this with \eqref{eq:dotgamma}, we obtain the desired bound.
\end{proof}

\section{Proof of  Theorem \ref{thm}} \label{S:proof}

From Corollary \ref{cor:LWP} and Lemma \ref{lem:energy}, the existence of a global solution to \eqref{sp} follows from the a priori bound  \eqref{est:u_infty}.
The case of $0<t<1$ is a consequence of Corollary \ref{cor:LWP} and the smallness of the initial data.
We therefore consider the time interval $[1,T]$, for which we make the bootstrap assumption
\[
|u(t,x)| + |u_x(t,x)|\le D \eps t^{-\frac{1}{2}} .
\]
Here, $D$ is chosen with $1 \ll D \ll \eps^{-\frac{1}{2}}$.
From Lemma \ref{lem:energy} and Proposition \ref{prop:est|u|}, we have
\begin{align} \label{eq:u_decay}
& |u (t,vt)|
\lesssim \eps t^{-\frac{1}{2} + D_{\ast} \eps} \left\{ \min ( |v|^{\frac{s}{4}-\frac{1}{2}}, |v|^{-\frac{3}{4}} ) + t^{-\frac{s-2}{2s+2}} ( 1+ \log t ) \right\} , \\
\notag
& |u_x (t,vt)|
\lesssim \eps t^{-\frac{1}{2} + D_{\ast} \eps} \left\{ \min ( |v|^{\frac{s}{4}-1}, |v|^{-\frac{5}{4}} ) + t^{-\frac{s-4}{2s+2}} ( 1+ \log t ) \right\} .
\end{align}
As in the previous section, we consider
\[
\Omega_{\alpha} (t) := \{ v \in \R_- : t^{-\alpha} \le -v \le t^{\alpha} \} ,
\]
where $\alpha$ is a fixed constant satisfying $0< \alpha < \min \left\{ \frac{2}{45}, \frac{2}{2s+1}, \frac{2(s-4)}{3(s+1)} \right\}$.
Outside $\Omega_{\alpha} (t)$, \eqref{est:u_infty} follows from the above bounds, provided that $s>4$ and $\eps >0$ is sufficiently small.

From Proposition \ref{prop:gamma_decay}, there exists a unique function $W$ defined on $\R_-$ such that for $t \ge 1$ and $v \in \Omega_{\alpha} (t)$,
\begin{equation} \label{gamma_approx}
\gamma (t, v) = W(v) e^{3i |v|^{-\frac{1}{2}} |W(v)|^2 \log t} + O (\eps t^{-\frac{1}{5}+3D_{\ast} \eps}) .
\end{equation}

Inside $\Omega_{\alpha} (t)$, from Proposition \ref{prop:approx}, we have
\begin{align}
& | u(t,vt) - 2 t^{-\frac{1}{2}} \Re \{ e^{i\phi (t,vt)} \gamma (t,v) \} |
\lesssim \eps t^{-\frac{1}{2}+D_{\ast} \eps} \left\{ t^{-\frac{1}{4}+\frac{5}{8}\alpha} + t^{-\frac{s-2}{2s+2}} (1+ \log t) \right\} , \label{eq:byProp4.2} \\
&
\begin{aligned}
& | u_x (t,vt) - 2 t^{-\frac{1}{2}} |v|^{-\frac{1}{2}} \Re \{ i e^{i\phi (t,vt)} \gamma (t,v) \} | \\
& \hspace*{140pt} \lesssim \eps t^{-\frac{1}{2}+D_{\ast} \eps} \left\{ t^{-\frac{1}{4} + \frac{9}{8} \alpha} + t^{-\frac{s-4}{2s+2}} (1+ \log t) \right\} .
\end{aligned} \notag
\end{align}
Therefore, it is sufficient to show that 
\[
| \gamma (t,v) | \lesssim \eps (1+|v|^{-\frac{1}{2}})^{-1}
\]
for any $v \in \Omega_{\alpha} (t)$.
The crucial point here is that the implicit constant does not depend on $D$.

For $|v| \sim 1$, by \eqref{eq:byProp4.2}, we have
\[
| \gamma (t,v) | \lesssim \eps
\]
for $t \sim 1$, which implies that $|W(v)| \lesssim \eps$.
For $t \gg 1$, by \eqref{gamma_approx}, we have
\[
| \gamma (t,v) |
\lesssim |W(v)| + \eps \lesssim \eps .
\]

When $|v| \ll 1$, let $t_0>1$ be $|v|=t_0^{-\alpha}$.
The estimates \eqref{eq:u_decay} and \eqref{eq:byProp4.2} yield
\[
|\gamma (t_0,v)|
\lesssim \eps t_0^{D_{\ast} \eps} |v|^{\frac{s}{4}-\frac{1}{2}}
\]
as $0< \alpha < \frac{2}{2s+1}$.
Solving the ordinal differential equation in Proposition \ref{prop:gamma_decay} with initial data $t=t_0$, we have
\begin{align*}
|\gamma (t,v) |
& \lesssim \eps t_0^{D_{\ast} \eps} |v|^{\frac{s}{4}-\frac{1}{2}} + \int_{t_0}^{\infty} t'^{-\frac{6}{5}} \eps t'^{3D_{\ast} \eps} dt' \\
& \lesssim \eps \left( t_0^{D_{\ast} \eps} |v|^{\frac{s}{4}-\frac{1}{2}} + t_0^{-\frac{1}{5} + 3D_{\ast} \eps} \right)
\lesssim \eps |v|^{\frac{1}{2}} ,
\end{align*}
provided that $s>4$ and $\eps>0$ is sufficiently small.

When $|v| \gg 1$, let $t_0>1$ be $|v|=t_0^{\alpha}$.
The estimates \eqref{eq:u_decay} and \eqref{eq:byProp4.2} yield
\[
|\gamma (t_0,v)|
\lesssim \eps t_0^{D_{\ast} \eps} |v|^{-\frac{3}{4}}
\]
as $0< \alpha < \frac{2(s-2)}{3(s+1)}$.
Solving the ordinal differential equation in Proposition \ref{prop:gamma_decay} with initial data $t=t_0$, we have
\begin{align*}
|\gamma (t,v) |
& \lesssim \eps t_0^{D_{\ast} \eps} |v|^{-\frac{3}{4}} + \int_{t_0}^{\infty} t'^{-\frac{6}{5}} \eps t'^{3D_{\ast} \eps} dt' \\
& \lesssim \eps \left( t_0^{D_{\ast} \eps} |v|^{-\frac{3}{4}} + t_0^{-\frac{1}{5} + 3D_{\ast} \eps} \right)
\lesssim \eps
\end{align*}
provided that $\eps>0$ is sufficiently small.

Finally, the modified scattering follows from \eqref{gamma_approx} and \eqref{eq:byProp4.2}.

\appendix

\section{Remark on the paper by Hayashi and Naumkin \cite{HayNau15}} \label{appendix}

In this appendix, we take $\delta =1$.
We denote the free propagator by $U(t)$, i.e., $U(t) := e^{t \partial_x^{-1}}$.
Lemma 3.3 in \cite{HayNau15} says that for $0< \rho < \frac{1}{2}$,
\begin{equation} \label{eq:HN,lem3.3}
\| \lr{i\partial_x} \phi \|_{L^{\infty}} \lesssim t ^{-\frac{1}{2}} \| x \partial_x U(-t) \phi \|_{L^2}^{\frac{1}{2}+\rho} \| U(-t) \phi \|_{H^{\frac{2-2 \rho}{1-2 \rho}}}^{\frac{1}{2}-\rho} + t^{-\frac{1}{2}} \| U(-t) \phi \|_{H^{\frac{5}{2}}}
\end{equation}
holds, provided that the right hand side is finite.
However, this estimate fails.
More precisely, we give a counterexample to the inequality in which $\| \lr{i\partial_x} \phi \|_{L^{\infty}}$ on the left hand side is replaced by $\| \partial_x \phi \|_{L^{\infty}}$.

Let $\chi \in C_0^{\infty} (\R)$ be a smooth function with $\supp \chi \subset [ -1,1]$.
We set
\[
\phi (x) := \mathcal{F}^{-1} \left[ \chi \left( \frac{\cdot -2N}{N} \right) \right] (x)
= N e^{i2Nx} \mathcal{F}^{-1} [\chi] (Nx)
\]
for sufficiently large $N \ge 1$.
Then,
\[
\| U(-t) \phi \|_{H^s} = \| \phi \|_{H^s} \lesssim N^{s+\frac{1}{2}}
\]
for any $s \ge 0$.
Because
\[
\mathcal{F}[ x U(-t) \phi ] (\xi)
= e^{-\frac{t}{i\xi}} \left( \frac{t}{\xi^2} \chi \left( \frac{\xi -2N}{N} \right) +i \frac{1}{N} \chi' \left( \frac{\xi -2N}{N} \right) \right) ,
\]
we have
\[
\| x \partial_x U(-t) \phi \|_{L^2}
\le \| \partial_x (x U(-t) \phi) \|_{L^2} + \| U(-t) \phi \|_{L^2}
\lesssim t N^{-\frac{1}{2}} + N^{\frac{1}{2}} .
\]
Accordingly, the right hand side of \eqref{eq:HN,lem3.3} is bounded by
\[
t^{-\frac{1}{2}} ( t N^{-\frac{1}{2}} + N^{\frac{1}{2}}) ^{\frac{1}{2}+ \rho} N^{\left( \frac{2-2\rho}{1-2\rho} + \frac{1}{2} \right) \left( \frac{1}{2} - \rho \right)} + t^{-\frac{1}{2}} N^3
\lesssim
t^{\rho} N^{1-2\rho} + t^{-\frac{1}{2}} N^3 .
\]
Therefore, setting $t = N^{2+\frac{1}{2\rho}}$, we obtain
\[
\text{R.H.S. of \eqref{eq:HN,lem3.3}} \lesssim N^{\frac{3}{2}} + N^{2-\frac{1}{4\rho}} .
\]
Conversely, from $\| \partial_x \phi \|_{L^{\infty}} \sim N^2$, the left hand side of \eqref{eq:HN,lem3.3} is bounded below by $N^2$.

This counterexample is a reflection of the fact that the regularity $\frac{2-2\rho}{1-2\rho}$ is very small.
Hence, we need to replace $\frac{2-2\rho}{1-2\rho}$ on the right hand side of \eqref{eq:HN,lem3.3} by $\frac{4-2\rho}{1-2\rho}$, which is reduced to Lemma 2.3 with $l=1$ in \cite{HNN13}.

If we naively use the estimate
\begin{equation} \label{eq:HN,lem3.3'}
\| \lr{i\partial_x} \phi \|_{L^{\infty}} \lesssim t ^{-\frac{1}{2}} \| x \partial_x U(-t) \phi \|_{L^2}^{\frac{1}{2}+\rho} \| U(-t) \phi \|_{H^{\frac{4-2 \rho}{1-2 \rho}}}^{\frac{1}{2}-\rho} + t^{-\frac{1}{2}} \| U(-t) \phi \|_{H^{\frac{5}{2}}}
\end{equation}
instead of \eqref{eq:HN,lem3.3}, then we need to replace the regularity conditions of the norm
\[
\| u (t) \| _{Y_T} = \sup _{t \in [0,T]} \lr{t}^{-\gamma} \left( \| u(t) \|_{\dot{H}^{-1}} + \| u(t) \|_{H^m} + \| \partial_x J u(t) \|_{H^l} \right),
\]
namely $m>\frac{5}{2}+l$ and $l> \frac{3}{2}$, by $m>4+l$ and $l> \frac{3}{2}$.
See the proof of Lemma 3.4 in \cite{HayNau15}.
More precisely, \eqref{eq:HN,lem3.3} is used to estimate $\| u_{xx} \|_{L^{\infty}}$ and $\| | \partial_x| ^s u_{xx} \|_{L^{\infty}}$ in the proof of Lemma 3.4 in \cite{HayNau15}.
For example, if we use \eqref{eq:HN,lem3.3'} to estimate $\| | \partial_x| ^s u_{xx} \|_{L^{\infty}}$ for $0<s<\frac{1}{2}$, we get
\[
\| | \partial_x| ^s u_{xx} \|_{L^{\infty}}
\lesssim t ^{-\frac{1}{2}} \| | \partial_x| ^s \partial_x^2 J u \|_{L^2}^{\frac{1}{2}+\rho} \| u \|_{H^{s+1+\frac{4-2 \rho}{1-2 \rho}}}^{\frac{1}{2}-\rho} + t^{-\frac{1}{2}} \| u \|_{H^{s+\frac{7}{2}}} .
\]
Therefore, it seems that the assumption $u_0 \in \dot{H}^{-1} \cap H^m$, $x \partial_x u_0 \in H^l$ with $m>\frac{5}{2}+l$ and $l>\frac{3}{2}$ must be to replaced by $m>4+l$ and $l>\frac{3}{2}$, as $s+1+\frac{4-2 \rho}{1-2 \rho} > s+5$.

\end{document}